\documentclass[11pt]{article}
\usepackage{hyperref}
\hypersetup{
    colorlinks=true,
    citecolor=green,
    linkcolor=green,
    filecolor=magenta,      
    urlcolor=cyan,
}

\usepackage{setspace}
\usepackage[margin=1cm]{caption}
\usepackage[cmyk]{xcolor}
\usepackage{graphicx}
\usepackage{float}
\usepackage{enumitem}
\usepackage[bottom]{footmisc}
\usepackage[titletoc,title]{appendix}

\graphicspath{ {images/} }
\usepackage{mathrsfs}

\usepackage{subdepth}
\usepackage{bm}
\usepackage{cite}

\usepackage{environ}
\usepackage{xparse}
\usepackage[margin=1in]{geometry}
\usepackage{amssymb,amsmath,amsthm}
\numberwithin{equation}{section}
\usepackage[numbers,sort]{natbib}
	\theoremstyle{definition}
	\newtheorem{definition}{Definition}[section]
    \newtheorem{example}{Example}[section]
    \newtheorem{remark}{Remark}[section]
	\theoremstyle{theorem}
    \newtheorem{corollary}{Corollary}[section]
	\newtheorem{theorem}{Theorem}[section]
    \newtheorem{lemma}{Lemma}[section]
       
	\theoremstyle{remark}
	
    \usepackage{thmtools}
	\usepackage{thm-restate}
    \usepackage[capitalise]{cleveref}

\NewEnviron{eqsplit}{%
  \begin{equation}\begin{split}
    \BODY
  \end{split}\end{equation}
}
\NewEnviron{eqsplit*}{%
  \begin{equation*}\begin{split}
    \BODY
  \end{split}\end{equation*}
}

\newcommand{\ds}{\displaystyle}
\newcommand\numberthis{\addtocounter{equation}{1}\tag{\theequation}}
\newcommand{\R}{\ensuremath{\mathbb{R}}}
\newcommand{\C}{\ensuremath{\mathbb{C}}}

\newcommand{\Z}{\ensuremath{\mathcal{Z}}}

\newcommand{\OO}[2]{\ensuremath{\left(#1,#2\right)}}

\NewDocumentCommand{\integ}{m m O{} D<>{}}{
\IfNoValueTF{#3}{\ensuremath{\displaystyle \int_{#2} #4 \, \mathrm{d}#1}}
{\ensuremath{\displaystyle \int_{#2}^{#3} #4 \, \mathrm{d}#1}}
}
\NewDocumentCommand{\doubintax}{m m O{} D<>{}}{
\IfNoValueTF{#3}{\ensuremath{\displaystyle \integ{c}{0}[\infty]<\xi(c)\left[\int_{#2} #4 \, \mathrm{d}#1\right]>}}
{\ensuremath{\displaystyle \integ{c}{0}[\infty]<\xi(c)\left[\int_{#2}^{#3} #4 \, \mathrm{d}#1\right]>}}
}

\NewDocumentCommand{\doubintfb}{m m O{} D<>{}}{
\IfNoValueTF{#3}{\ensuremath{\displaystyle \integ{\tau}{0}[\infty]<\eta(\tau)\left[\int_{#2} #4 \, \mathrm{d}#1\right]>}}
{\ensuremath{\displaystyle \integ{\tau}{0}[\infty]<\eta(\tau)\left[\int_{#2}^{#3} #4 \, \mathrm{d}#1\right]>}}
}
\NewDocumentCommand{\normLinf}{m o}
{
\IfNoValueTF{#2}
{\ensuremath{\left\lVert #1 \right\rVert_{L^\infty (\R)}}}{\ensuremath{\left\lVert (#1,#2) \right\rVert_{L^\infty (\R)}}}
}
\newcommand{\e}[1]{\ensuremath{\exp{\left(#1\right)}}}
\let\RealFrac\frac
\NewDocumentCommand\f{mg}{%
  \IfNoValueTF{#2}
    {\RealFrac{1}{#1}}
    {\RealFrac{#1}{#2}}%
}
\newcommand{\K}[1]{\ensuremath{K\left(#1\right)}}

\title{
{Traveling Wave Solutions to a Neural Field Model With Oscillatory Synaptic Coupling Types}\\
	}
	\author{Alan Dyson\footnote{Department of Mathematics, Lehigh University, 14 East Packer Avenue, Bethlehem, PA 18015 (acd313@lehigh.edu)}}
	\date{\today}
    
\begin{document}
\maketitle
%Abstract %226 words
\abstract 
In this paper, we investigate the existence, uniqueness, and spectral stability of traveling waves arising from a single threshold neural field model with one spatial dimension, a Heaviside firing rate function, axonal propagation delay, and biologically motivated oscillatory coupling types. Neuronal tracing studies show that long-ranged excitatory connections form stripe-like patterns throughout the mammalian cortex; thus, we aim to generalize the notions of pure excitation, lateral inhibition, and lateral excitation by allowing coupling types to spatially oscillate between excitation and inhibition. In turn, we hope to analyze traveling fronts and pulses with novel shapes. With fronts as our main focus, we exploit Heaviside firing rate functions in order to establish existence and utilize speed index functions with at most one critical point as a tool for showing uniqueness of wave speed. We are able to construct Evans functions, the so-called stability index functions, in order to provide positive spectral stability results. Finally, we show that by incorporating slow linear feedback, we can compute fast pulses numerically with phase space dynamics that are similar to their corresponding singular homoclinical orbits; hence, our work answers open problems and provides insight into new ones.\\ \\
%In this paper, we investigate the existence, uniqueness, and spectral stability of traveling waves arising from a single threshold neural field model with one spatial dimension, a Heaviside firing rate function, axonal propagation delay, and biologically motivated oscillatory coupling types. Neuronal tracing studies show that long-ranged excitatory connections form stripe-like patterns throughout the mammalian cortex; thus, we aim to generalize the notions of pure excitation, lateral inhibition, and lateral excitation by allowing coupling types to spatially oscillate between excitation and inhibition. In turn, we hope to analyze traveling fronts and pulses with novel shapes. With fronts as our main focus, we exploit Heaviside firing rate functions in order to establish existence and utilize speed index functions with at most one critical point as a tool for showing uniqueness of wave speed. Since fronts arising from Heaviside firing rate functions have closed-form solutions, we are able to easily construct Evans functions, the so-called stability index functions, in order to study spectral stability. Finally, we show that by incorporating slow linear feedback and moving to a system of integral differential equations, we can compute fast pulses numerically with phase space dynamics that are remarkably similar to their corresponding singular homoclinical orbits; hence, our work has connections to mathematical neuroscience as well as perturbation theory.\\ \\

%SECTIONS commented out if combined
%\input{Sections/Section1_Introduction}
%\input{Sections/Section2_Kernels}
%\input{Sections/Section3_EU}
%\input{Sections/Section4_stab}
%\input{Sections/Section5_Pulses}
%\input{Sections/Discussion}
%Section: INTRODUCTION
\section{Introduction} %Introduction

Modern research of the mammalian brain has significantly gravitated towards predicting, observing, and analyzing traveling waves.
The combination of using sophisticated electrode recording technology and pharmacologically blocking inhibition allows researchers to observe these patterns experimentally \cite{Golomb2011,traub1993analysis}. Pathology and general physiological phenomenon are often strong motivators for such research; for example, traveling waves have been observed during epileptiform \cite{Connors-Generationofepileptiform,traub1993analysis,Golomb2011}, migraines \cite{Lance-CurrentConceptsofmigraine}, and visual stimuli \cite{Lee2005,Sato2012,Benucci2007}. Naturally, computational and theoretical mathematical modeling arises in order to predict or explain propagations. 

We gain insight as to why mathematical models have been proposed for decades by realizing that the mammalian nervous system is extraordinarily complex. For some perspective, comprehensive studies show that in a human neocortex, there are approximately 20 billion neurons and $.15 \times 10^{15}$ synapses \cite{Bio:Agingandthehumanneocortex}. Even a single synaptic event, spanning from an action potential to neurotransmission, is highly nontrivial to model. Notably, the famous Hodgkin Huxley experiments \cite{hodgkin1952quantitative} provided a basis for understanding the connection between action potentials and voltage-gated sodium and potassium channels. More models have been implemented in order to track the impact that neurotransmitters have on conductance after reaching their postsynaptic receptors. Changes caused by excitatory and inhibitory neurotransmitters lead to fast and slow dynamics based on receptor type \cite{bressloff2014waves}. By adding firing rate patterns such as oscillations and bursting \cite{Izhikevich:2006} into the mix, we see that single neuron dynamics constitute a deep segment of neuroscience in their own right.

The difficulty level grows exponentially when we embed single cell dynamics into networks. After accounting for metabolic processes like spike frequency adaptation and synaptic depression at the network level, we may convince ourselves that macroscopic modeling has its place for capturing and predicting novel wave-like behaviors. While we acknowledge that there are different valid approaches to the problem, we choose to treat firing times of single neurons as uncorrelated \cite{TimeStructure} and consider a neural field model, treating space and time as continuous. Fittingly, we model firing rates as functions at the network level that only depend on the average voltage in spatial patches. Since average voltage is implicitly related to single neuron dynamics, we find that single neuron features are relevant, but significantly less emphasized.

% In order to partially mitigate the downside of losing precise single neuron information, we take advantage of the known layering and coupling formations that exists in cortical networks. For example, in \cite{SynapticTargets,Topographyofpyr}, biotinylated dextran amine was injected in the layer 3 prefrontal cortex of macaque monkeys and revealed long-range stripe like connections across pyramidal neurons with GABAergic interneurons in the gaps. 

Instead, our emphasis is on how patches of neurons connect based on spatial positioning. A key component of our model is invoking a homogeneous synaptic coupling weight kernel $K(x-y)$ to describe the spatial contribution to membrane potential that presynaptic neurons at position $y$ contribute to postsynaptic neurons at position $x$. Here the sign of $K(x-y)$ determines whether the connection is inhibitory or excitatory and the magnitude determines the connection strength.

Coupling patterns in the neocortex described by oscillations of excitation and inhibition are known to exist and therefore, merit investigation. In particular, superficial layers such as layers 2 and 3 are where excitatory pyramidal cells make extensive arborizations laterally with inhibitory interneurons in the gaps \cite{Douglas2004}. Depending on the mammal and the brain region, tracing studies reveal excitatory stripes of varying patterns and purposes. For example, tracer injections in the visual cortex of cats reveal intercolumn connections are likely based on functional specificity \cite{Columnar}; intracortical connections are correlated with visual experiences \cite{Lowel1992} and context-dependent processing \cite{Martin2017}. Similar horizontal connections have also been found in the primary visual cortexes of tree shrews \cite{Chisum2003} and macaque monkeys \cite{Lund1993}. Experiments have also been carried on on the macaque monkey prefrontal cortex where biotinylated dextran amine was injected in the layer 3 prefrontal cortex of macaque monkeys and revealed long-range stripe-like connections \cite{SynapticTargets,Topographyofpyr,Lund1993}.

Motivated by such biologically observed connection types, we seek traveling wave solutions arising from well-studied neural field models. Others have invoked similar coupling types in related work on traveling and standing waves arising from models like ours; see \cite{Layer3,PDEmethods,Multiplebumps-Laing,LijunZhangetal,LijZhangSolo,LvWang,LaingTroy,Botelho2008}.

\subsection{Model Equations and Background}
\label{subsec: Model Equations}
In this paper, fronts are our main focus. Thus, we study the following homogeneous neural field model with axonal transmission delay \cite{PintoandErmentrout-SpatiallyStructuredActivityinSynapticallyCoupledI.TravelingWaves,ExistenceandUniqueness-ErmMcLeod}:

\begin{equation} \label{eq:paper_main}
u_t + u = \alpha \integ{y}{\R}[]<K(x-y)H(u\left(y,t-\f{c_0}|x-y|\right)-\theta)>,
\end{equation}
where $u = u(x,t)$ is the mean electric potential in the spatial patch at position $x$ and time $t$. The aforementioned kernel function $K$ represents the coupling strength and type (excitatory or inhibitory) based on spatial positioning. The constant parameter $\alpha > 0$ controls synaptic coupling strength, while $\theta > 0$ is the single threshold of excitation for the network. A delay in transmission arises from the parameter $c_0$, which is the speed of action potentials in the network. The Heaviside step function $H$ represents the firing rate of a single neuron. In this simplified model, we can understand our firing rate function as a binary mechanism: if $u(x,t)$ is above a threshold $\theta$, the neurons in spatial patch $x$ will fire at a maximum rate; otherwise, they will not fire at all. Such an assumption makes sense in biology and is used to simplify the analysis.

In order to understand the biophysical basis of \eqref{eq:paper_main}, we briefly review the derivation in \cite{ExistenceandUniqueness-ErmMcLeod}. First suppose we have a fixed spatial patch of postsynaptic neurons at position $x$ and $N$ arbitrarily spaced presynaptic patches $y_i$ for $i = 1,..,N$. Moreover, suppose we observe activity at times $s_j$ for $j =1,...,M$. For each $i$ and $j$, activity from patch $y_i$ at time $s_j$ contributes
$$
\alpha \eta(t-s_j)K(x-y_i)S(u\left(y_i,s_j-\f{c_0}|x-y_i|\right)-\theta)\Delta y_i \Delta s_j
$$
to the potential of neurons at patch $x$. Here $\eta(t-s)$ is the time dependent contribution that activity at time $s$ has at time $t$, $S$ is the firing rate for the network where axonal velocity delay is properly accounted for. Summing over all $i$ and $j$, the total input to neurons at position $x$ is
\begin{equation}
\label{eq:sum_potential}
\alpha \sum_{j=1}^M \sum_{i=1}^N \eta(t-s_j)K(x-y_i)S(u\left(y_i,s_j-\f{c_0}|x-y_i|\right)-\theta)\Delta y_i \Delta s_j. 
\end{equation}
If we allow $M$ and $N$ to go to infinity and suppose \eqref{eq:sum_potential} is the only input the neurons at $x$ receive at time $t$, $u(x,t)$ takes on the form
\begin{equation}
\label{eq: original_main}
u(x,t) = \alpha \int_{-\infty}^t \eta(t-s)\int_\R K(x-y)S(u\left(y,s-\f{c_0}|x-y|\right)-\theta) \, \mathrm{d}y\mathrm{d}s.
\end{equation}
Finally, for the special case where $\eta (t) = \e{-t}H(t)$ and $S$ is the Heaviside step function, differentiation of \eqref{eq: original_main} leads to \eqref{eq:paper_main}. Note that other choices for $\eta$ may be implemented based on neurotransmitter dynamics.

Research concerning numerous wave form solutions of \eqref{eq:paper_main} has expanded largely because of some important initial results. Wilson and Cowan \cite{WilsonCowan1972} were the first to study space and time coarse graining of local populations of interacting excitatory and inhibitory neurons. Amari \cite{Amari1977} was the first to biologically justify simplifying the Wilson-Cowan model to one that combines excitatory and inhibitory neurons; by doing so, he implemented Heaviside firing rates to obtain closed form bump (respectively, traveling front) solutions when Mexican hat (respectively, pure excitation) kernels were used. In \cite{ExistenceandUniqueness-ErmMcLeod}, Ermentrout and McLeod applied a homotopy argument to prove the existence of traveling fronts in the presence of sigmoidal firing rate functions and nonnegative kernels. Pinto and Ermentrout \cite{PintoandErmentrout-SpatiallyStructuredActivityinSynapticallyCoupledI.TravelingWaves} were the first to add slow linear feedback and derive a singular perturbation problem to obtain traveling pulses. 

Following the seminal, foundational work above, there have been hundreds of related studies. Some of the primary interests are comparing wave speeds in models to experiments, bifurcations, threshold noise, single and multiple standing and traveling pulses, spectral theory, operator theory methods, heterogeneities, multiple thresholds, and synaptic depression. For comprehensive background on the models, see \cite{CoombesOwenLord-WavesandBumps,Ermentrout-NeuralNetworks,FoundationsofMathNeurosicnece-ErmentroutTerman,bressloff2014waves} and the sources within.

\subsection{Main Goal and Improvement From Previous Results}
\label{subsec: Main_goal}
Our main result, accomplished in \cref{sec: analysis_Existence,sec: spectral_stability}, are the following two theorems concerning existence, uniqueness, and stability when $K$ is in one of the kernel classes formulated in \cref{subsec: ker_classes}.

\begin{theorem}[Existence and Uniqueness of Front]
\label{thm: E_and_U}

Suppose that $0<2\theta<\alpha$ and  $K$ is in class $\mathcal{A}_{j,k}$, $\mathcal{B}_{j,k}$, or $\mathcal{C}_{j,k}$ for some integers $j$ and $k$. Then there exists a unique traveling wave front solution $u(x,t) = U(z)$ to \eqref{eq:paper_main} such that $U(0) = \theta$, $U'(0) > 0$, $U(z) < \theta$ on $(-\infty,0)$, and $U(z) > \theta$ on $(0,\infty)$. The front satisfies the reduced equation
\begin{equation} \label{eq:main_reduced}
\mu_0 U' + U = \alpha \integ{y}{\R}[]<K(z-y)H\left(U\left(y-\f{\mu_0}{c_0}|z-y|\right)-\theta\right)>
\end{equation}
with exponentially decaying limits
\begin{align*}
\lim_{z \to -\infty} U(z) = 0, \qquad \lim_{z \to \infty} U(z) = \alpha, \qquad \lim_{z \to \pm\infty} U'(z) = 0.
\end{align*}
The wave travels under the traveling coordinate $z = x + \mu_0 t$ at the unique wave speed $\mu_0 \in (0,c_0)$.
\end{theorem}
\begin{theorem}[Stability of Front]
\label{thm: stability}
Suppose that $0<2\theta<\alpha$, $K$ is in class $\mathcal{A}_{j,k}$, $\mathcal{B}_{j,k}$, or $\mathcal{C}_{j,k}$ for some integers $j$ and $k$, and $U$ is a unique solution described in \cref{thm: E_and_U}. If the Laplace transform of $K$ satisfies
\begin{equation}\label{eq: Laplace}
\integ{x}{-\infty}[0]<\e{s x}K(x)>=\f{p(s)}{q(s)},
\end{equation}
where $p$ and $q$ are polynomials of degree at most two, then $U$ is spectrally stable. Moreover, if $c_0=\infty$, then $U$ is linearly and nonlinearly stable.
\end{theorem}
The resulting solutions $U(\cdot)$ are heteroclinical orbits that connect the fixed points $U\equiv 0$ and $U\equiv \alpha$, crossing the threshold $\theta$ exactly once. While this idea has been investigated before, we improve previous results.

Firstly, we expand the class of kernels such that \eqref{eq:paper_main} has traveling wave solutions with unique wave speeds. By translation invariance, we assume without loss of generality, that $U(0) = \theta$. In turn, we handle the unique wave speed problem using a speed index function, first derived by Pinto and Ermentrout \cite{PintoandErmentrout-SpatiallyStructuredActivityinSynapticallyCoupledI.TravelingWaves} and later by others:

\begin{equation} 
\label{eq: speed_index}
\phi(\mu) := \integ{x}{-\infty}[0]<\e{\f{c_0-\mu}{c_0\mu}x}K(x)>.
\end{equation}
Our main improvement can be seen in our handling of unique roots $\mu_0 \in (0,c_0)$ of the compatibility equation
\begin{equation} \label{eq:step2}
\phi(\mu) = \f{2} - \f{\theta}{\alpha},
\end{equation}
for $0 <2\theta < \alpha$, arising from the requirement that the traveling waves cross the threshold exactly once.
Under our methods, we consider three new biologically motivated types of oscillatory kernel classes. The first class type can be understood as a result of combining the mechanics of previous results. The other two class types invoke new techniques; they are constructed with the intent of showing that $\phi$ has at most one critical point.

In \cref{sec: spectral_stability} we formulate the eigenvalue problem and examine spectral stability of our solutions for kernels of all three types. The essential spectrum is shown to be on the left half plane and for the point spectrum analysis on the right half plane, our main tool is the complex analytic Evans function \cite{Evans-I,Evans-II,Evans-III,Evans-IV}
\begin{equation} \label{eq:Evans_main}
\mathcal{E}(\lambda) := 1-\f{\phi(\f{\mu_0}{\lambda + 1})}{\phi(\mu_0)},
\end{equation}
the so-called stability index function, where roots are equivalent to eigenvalues. Although many authors have used the Evans function successfully, stability of traveling waves has not been discussed concerning many of the kernels considered in this paper. As a special case, we exploit the uniqueness of the wave speed to prove \cref{thm: stability}.

In \cref{sec: Example_K2}, we show how to apply our results to a kernel type commonly studied in the literature,
$$
K(x)=C(a)\e{-a|x|}(a\sin(|x|)+\cos(x)),
$$
by fully classifying the existence, uniqueness, and stability of front solutions to such a model as a function of $a$ and $\theta$.

We conclude with \cref{sec: pulses}, where we motivate future work by computing fast traveling pulses arising from the singularly perturbed system of integral equations \cite{PintoandErmentrout-SpatiallyStructuredActivityinSynapticallyCoupledI.TravelingWaves,Pinto2005}:
\begin{align}
u_t + u + w & = \alpha \integ{y}{\R}[]<K(x-y)H(u\left(y,t-\f{c_0}|x-y|\right)-\theta)>,  \label{eq: pulse1}\\
w_t & = \epsilon(u-\gamma w), \label{eq: pulse2}
\end{align}
for $0<\epsilon \ll 1$. Using the same example kernels used for the fronts, pulses are plotted and compared to singular solutions in phase space portraits.

\subsubsection{Mathematical Difficulties and Open Problems}\label{subsubsec: math_diff}
In the case of Heaviside firing rates and traveling front solutions to scalar models like \eqref{eq:paper_main}, we remark that existence, uniqueness, and stability results become deceptively difficult to prove when we allow $K$ to have any regions of negative output. The setup of the problem and canonical tools involved -- like speed \eqref{eq: speed_index} and stability \eqref{eq:Evans_main} index functions -- are almost exactly the same across studies. But the difference in difficulty jumps astronomically when $K\geq 0$ does not hold since standard methods like monotonicity arguments, Gr\"{o}nwall's inequality, and maximum principle no longer hold in a straightforward and meaningful way.

For example, when $K\geq 0$, it is trivial to see that $\phi$ is strictly increasing and \eqref{eq:step2} holds for some unique $\mu_0$. Now consider when $K$ crosses the $x$-axis countably many times. With the exception of lateral excitation kernels \cite{Zhang-HowDo}, all previous work imposes conditions on $K$ to replicate these mechanics. To the author's knowledge, relaxing assumptions in such a way that $\phi$ has critical points is new when $K$ is general.

In the stability analysis, we can also see illustrative and generalizable reasons why our work is highly nontrivial. When $K\geq 0$, one can show that the ubiquitous inequality
\begin{equation}\label{eq: Evans_max}
\left\lvert \integ{x}{-\infty}[0]<\e{\f{c_0(\lambda+1)-\mu}{c_0\mu}x}K(x)>\right\rvert < \integ{x}{-\infty}[0]<\e{\f{c_0-\mu}{c_0\mu}x}K(x)> 
\end{equation}
holds when Re$(\lambda) \geq 0$, $\lambda \neq 0$, showing that $|1-\mathcal{E}(\lambda)| < 1$. Maximum principle is then the main tool for proving spectral stability. In contrast, when $K$ has countably many zeros, proving inequalities like \eqref{eq: Evans_max} are not straightforward at all and they are not always necessary.  In fact, in \cite{Zhang-HowDo}, which to the author's knowledge contains the most rigorous stability analysis of fronts arising from kernels with only two zeros on the real line, a rigorous proof of \eqref{eq: Evans_max} is not provided and is still an open problem. Without imposing strict assumptions on $K$, we expect such a problem to be difficult to resolve -- especially considering one has to simultaneously prove existence and uniqueness by calculating $\mu$.

We close this section by remarking that by \cref{cor: EU_MH}, the present work completes the proof of the existence and uniqueness of fronts for all Mexican hat kernels with the only assumption being that $K$ is exponentially bounded. Therefore, the most natural place to begin proving  \eqref{eq: Evans_max} is for Mexican hat kernels, which we leave to the reader.

%SECTION: KERNEL CONSTRUCTIONS
\section{Kernel Classes} \label{sec: kernel_constructions}
In this section, we systematically describe the oscillatory kernel classes we wish to study; examples are then provided in \cref{subsec: ker_classes}.
\subsubsection*{Basic Assumptions and Terminology} \label{Basic_assumptions}
In all cases, we assume our kernel $K$ has the following typical properties for this type of problem:

\begin{equation*}
\int_{-\infty}^0 K(x) \, \mathrm{d}x = \int_{0}^{\infty} K(x) \, \mathrm{d}x = \frac{1}{2}, \qquad |K(x)|\leq C\e{-\rho |x|} \qquad \text{for all } x \in \R.
\end{equation*}
For some background, we define commonly used kernel types that oscillate at most once on each half plane.
\begin{definition} \label{def: PE}
We say $K$ is a {\it pure excitation} kernel if $K(x)\geq 0$ for all $x$.
\end{definition}

\begin{definition} \label{def: MH}
We say $K$ is a {\it lateral inhibition} kernel, or Mexican hat kernel, if there exists unique constants $M_1 > 0$ and $M_2 > 0 $ such that $K(x) \geq 0$ on $(-M_1,M_2)$ and $K(x) \leq 0 $ on $(-\infty,-M_1) \cup (M_2,\infty)$.
\end{definition}
\begin{definition} \label{def: UPMH}
We say $K$ is a {\it lateral excitation} kernel, or upside down Mexican hat kernel, if there exists unique constants $N_1 > 0$ and $N_2 > 0 $ such that $K(x) \geq 0$ on $(-\infty,-N_1)\cup (N_2,\infty)$  and $K(x) \leq 0$ on $(-N_1,N_2)$. 
\end{definition}

We will come back to these definitions in \cref{subsec: previous_kernels} after we have defined our oscillatory kernel classes.

\subsection{Wave Speed Conditions for Uniqueness}
\label{sec: unique_assumptions}
The conditions presented in this subsection, which we shall call {\it wave speed conditions}, underlie the most substantial difference in approach between this work and others. In particular, This subsection provides the groundwork for obtaining the existence and uniqueness of wave speeds when the kernel functions oscillate any number of times on the left half plane.
The following repeated integral of $|x|K(x)$, originally proposed in \cite{HuttZhang-TravelingWave}, will be used throughout the paper: 
\begin{equation}
\begin{split}
\Lambda^0 K(x) & := |x|K(x), \\
\Lambda^n K(x) & := \int_x^0\int_{x_{n-1}}^0...\int_{x_1}^0 |x_0| K(x_0) \, \mathrm{d}x_0...\mathrm{d}x_{n-2}\mathrm{d}x_{n-1}\\
& = \int_x^0 \Lambda^{n-1} K(x_{n-1}) \, \mathrm{d}x_{n-1} \quad \text{for } x\leq0 \text{ and } n\geq 1.
\end{split}
\end{equation}
Using this definition, we define three wave speed conditions.

\begin{itemize}
\item[($A_n$)]  Suppose there exists $N>0$ such that $\Lambda^N K(x) \geq 0$ for all $x\leq 0$. Then the same condition holds for all $m\geq N$. Denoting $n$ as the smallest $N$ where such a property holds, we say $K$ satisfies wave speed condition ($A_n$).

%Suppose there exists $n\geq 0$ such that $\Lambda^n K(x) \geq 0$ for all $x\leq 0$. Then we say $K$ satisfies wave speed condition ($A_n$).

\item[($B_n$)] Suppose there exists $N>0$ and a constant $B_N>0$ such that $\Lambda^N K(x) \geq 0 $ on $(-B_N,0)$, $\Lambda^N K(x) \leq (\not\equiv) \, 0$ on $(-\infty,-B_N)$. Then the same condition holds for all $m\geq N$. Denoting $n$ as the smallest $N$ where such a property holds, we say $K$ satisfies wave speed condition ($B_n$).

% Suppose there exists $n\geq 0$ and a constant $B_n>0$ such that $\Lambda^n K(x) \geq 0 $ on $(-B_n,0)$, $\Lambda^n K(x) \leq 0$ on $(-\infty,-B_n)$, and . Then we say $K$ satisfies wave speed condition ($B_n$).

\item[($C_n$)] Suppose there exists $N>0$ and a constant $C_N>0$ such that $\Lambda^N K(x) \leq (\not\equiv)\, 0 $ on $(-C_N,0)$, $\Lambda^N K(x) \geq 0$ on $(-\infty,-C_N)$. Then the same condition holds for all $m\geq N$. Denoting $n$ as the smallest $N$ where such a property holds, we say $K$ satisfies wave speed condition ($C_n$).

% Suppose there exists $n\geq 0$ and a constant $C_n>0$ such that $\Lambda^n K(x) \leq 0 $ on $(-C_n,0)$ and $\Lambda^n K(x) \geq 0$ on $(-\infty,-C_n)$. Then we say $K$ satisfies wave speed condition ($C_n$).
\end{itemize}
We remark that pure excitation kernels satisfy ($A_1$); lateral inhibition kernels satisfy ($A_1$)  (respectively ($B_1$)) if $\integ{x}{-\infty}[0]<|x|K(x)>\geq 0$ (respectively $<0$); lateral excitation kernels satisfy ($C_1$). Roughly speaking, the smaller $n$ is, the closer $K$ resembles pure excitation, lateral inhibition, or lateral excitation respectively.

Our strategy involving these conditions breaks down as follows:
\begin{itemize}
\item ($A_n$) results in $\phi$ strictly increasing.
\item ($B_n$) leads to $\phi$ strictly increasing when $\phi \in (0,\f{2})$ with  one local maximum when $\phi>\f{2}$.
\item ($C_n$) leads to $\phi$ strictly increasing when $\phi \in (0,\f{2})$ with at most one local minimum when $\phi<0$.
\end{itemize}
%subsection
\subsection{Threshold Requirements} \label{subsubsec: left_and_right}

The proceeding left and right half plane conditions, which we will call {\it threshold conditions}, ensure that regardless of how much $K$ oscillates, our traveling wave front solution satisfies $U(\cdot) < \theta$ on $(-\infty,0)$ and $U(\cdot) > \theta$ on $(0,\infty)$.

\subsubsection*{Left Half Plane Threshold Conditions}
We consider situations where $K$ transversely crosses the negative $x$-axis at most countably many times. The first proceeding condition $\mathcal{L}_{0}$ represents pure excitation on the left half plane. The next condition $\mathcal{L}_{j}$ was formulated from the original work in \cite[Assumptions $(L_2)-(L_3)$, p.~2-3]{LijunZhangetal}. We do, however, remove the requirements 
\begin{equation}
\label{eq: weight_change}
\integ{x}{-M_{2n}}[-M_{2n-2}]<|x|K(x)>\geq 0, \qquad \integ{x}{-M_{2n}}[0]<|x|K(x)>\geq 0
\end{equation}
for $1\leq n < \infty$, since these estimates are special cases of wave speed condition ($A_1$). 

Based on the methods in \cite{LijunZhangetal}, the requirements in \eqref{eq: weight_change} appear to be the main barrier that prevents their kernel classes from including kernels that oscillate infinitely many times. By modifying these requirements, we overcome this obstacle. We see such improvement in the creation of condition $\mathcal{L}_{\infty}$ below. Furthermore, we present the conditions $\mathcal{L}_{-j}$ and $\mathcal{L}_{-\infty}$, which for $j \neq 1$ are new and arise since we may use wave speed condition ($C_n$).

\begin{itemize}
\item[$\mathcal{L}_{0}$:] Suppose $K(x) \geq 0$ on $(-\infty,0)$. Then $K$ satisfies condition $\mathcal{L}_{0}$.
\item[$\mathcal{L}_{j}$:] Suppose $K$ transversely crosses the negative $x$-axis exactly $j$ times in the sense that there exists constants $0<M_1<M_2<...<M_j$
such that $K(-M_n) = 0$ and $sgn(K'(-M_n)) = (-1)^{n+1}$ for $n=1,2,...,j$. Also, if $j\geq 2$, suppose
\begin{align*}
\f{\alpha}{2} - \alpha\integ{x}{-M_{2n}}[0]<K(x)> < \theta \qquad \text{for } n = 1,2,...,\left\lfloor \f{j}{2} \right\rfloor.
\end{align*}
Then $K$ satisfies condition $\mathcal{L}_{j}$. If $K$ transversely crosses the negative $x$-axis infinitely many times, we allow $j\to \infty$.

\item[$\mathcal{L}_{-j}$:] Suppose $K$ transversely crosses the negative $x$-axis exactly $j$ times in the sense that there exists constants $0<M_1<...<M_{j}$
such that $K(-M_n) = 0$ and $sgn(K'(-M_n)) = (-1)^{n}$ for ${n=1,2,...,j}$. Also, if $j\geq 3$, suppose
\begin{align*}
\f{\alpha}{2} - \alpha\integ{x}{-M_{2n+1}}[0]<K(x)> < \theta \qquad \text{for } n = 1,2,...,\left\lfloor \f{j-1}{2} \right\rfloor.
\end{align*}
Then $K$ satisfies condition $\mathcal{L}_{-j}$. If $K$ transversely crosses the negative $x$-axis infinitely many times, we allow $j \to \infty$.
\end{itemize}
\begin{remark}
The main physical difference between $\mathcal{L}_{j}$ and $\mathcal{L}_{-j}$ can be understood in the following manner: For a fixed postsynaptic patch of neurons at position $x$, the local presynaptic neurons at position $y$ satisfying $x<y<x+M_1$ will be excitatory if $K(x-y)$ satisfies 
$\mathcal{L}_{j}$ and inhibitory if $K(x-y)$ satisfies $\mathcal{L}_{-j}$. A similar physical interpretation can be drawn from $\mathcal{R}_{k}$ and $\mathcal{R}_{-k}$ below, but with the position of local presynaptic neurons relative to $x$ reversed.
\end{remark}
%Right half plane
\subsubsection*{Right Half Plane Threshold Conditions}
Conditions $\mathcal{R}_{k}$ and $\mathcal{R}_{-k}$ below come from \cite[Assumption $(R_2)-(R_6)$, p.~3]{LijunZhangetal}.
\begin{itemize}
\item[$\mathcal{R}_{0}$] Suppose $K(x) \geq 0$ on $(0,\infty)$. Then $K$ satisfies condition $\mathcal{R}_{0}$.

\item[$\mathcal{R}_{k}$:] Suppose $K$ transversely crosses the positive $x$-axis exactly $k$ times in the sense that there exists constants $0<N_1<...<N_{k}$
such that $K(N_n) = 0$ and $sgn(K'(N_n)) = (-1)^{n}$ for $n=1,2,...,k$. Also, if $k\geq 2$, suppose
\begin{align*}
\f{\alpha}{2} + \alpha\integ{x}{0}[N_{2n}]<K(x)> > \theta \qquad \text{for } n = 1,...,\left\lfloor \f{k}{2} \right\rfloor.
\end{align*}
Then $K$ satisfies condition $\mathcal{R}_{k}$. If $K$ transversely crosses the positive $x$-axis infinitely many times, we allow $k \to \infty$.

\item[$\mathcal{R}_{-k}$:] Suppose $K$ transversely crosses the positive $x$-axis exactly $k$ times in the sense that there exists constants $0<N_1<...<N_{k}$
such that $K(N_n) = 0$ and $sgn(K'(N_n)) = (-1)^{n+1}$ for $n=1,...,k$. Also, suppose
\begin{align*}
\f{\alpha}{2} + \alpha\integ{x}{0}[N_{2n-1}]<K(x)> > \theta \qquad \text{for } n = 1,...,\left\lfloor \f{k+1}{2} \right\rfloor.
\end{align*}
Then $K$ satisfies condition $\mathcal{R}_{-k}$. If $K$ transversely crosses the positive $x$-axis infinitely many times, we allow $k \to \infty$.
\end{itemize}
\begin{remark}\label{remark: sym}
If $K$ is symmetric, as it often is assumed to be for these types of problems, and satisfies $\mathcal{L}_j$ for $j\geq1$, then $K$ also satisfies $\mathcal{R}_k$ for $k = j$.
\end{remark}

\subsection{Three Families of Kernel Classes} \label{subsec: ker_classes}

We define the following disjoint types of kernel classes used throughout this paper:

\begin{itemize}
\item[$\mathcal{A}_{j,k}$] Suppose on the left half plane, $K$ satisfies threshold condition $\mathcal{L}_{j}$ for some $j\geq0$ and wave speed condition ($A_n$) for some $n> 0$; on the right half plane $K$ satisfies threshold condition $\mathcal{R}_{k}$ for some $k\geq 0$. Then we say $K$ is in class $\mathcal{A}_{j,k}$.

\item[$\mathcal{B}_{j,k}$] Suppose on the left half plane, $K$ satisfies threshold condition $\mathcal{L}_{j}$ for some $j\geq 1$ and wave speed condition ($B_n$) for some $n> 0$; on the right half plane $K$ satisfies threshold condition $\mathcal{R}_{k}$ for some $k\geq 0$. Then we say $K$ is in class $\mathcal{B}_{j,k}$.

\item[$\mathcal{C}_{j,k}$] Suppose on the left half plane, $K$ satisfies threshold condition $\mathcal{L}_{-j}$ for some $j\geq 1$ and wave speed condition ($C_n$) for some $n> 0$; on the right half plane $K$ satisfies threshold condition $\mathcal{R}_{-k}$ for some $k\geq 0$. Then we say $K$ is in class $\mathcal{C}_{j,k}$.
\end{itemize}
\subsubsection*{Examples}
Before we proceed with proving our results, we consider three examples of kernels that oscillate countably many times. These kernels will be used throughout the paper to supplement our understanding of how the kernel classes connect to existence, uniqueness, and stability of wave front solutions to \eqref{eq:paper_main}. We also use the same example kernels when studying the pulse in \cref{sec: pulses}. In all cases, the kernels are symmetric and normalized by a constant $A$ in order to integrate to one. The rest of the model parameters are assigned the values $\alpha = 1$, $\theta = 0.4$, and $c_0 = 1$.
%Example 1
\begin{example}
The first example is a kernel in class $\mathcal{A}_{\infty,\infty}$:
\begin{equation}
K_1(x) := A \e{-a|x|}(\cos(bx)+c),
\end{equation}
where $a=0.2$, $b=2$, $c=0.4$.
\end{example}
%Example 2
\begin{example} \label{ex: ex_2}
The second example is a kernel in class $\mathcal{B}_{\infty,\infty}$:
\begin{equation}
K_2(x) := A \e{-a|x|}(a\sin(|x|)+\cos(x)),
\end{equation}
where $a=0.3$. This type of kernel has already been studied explicitly in the setting of standing waves in \cite{PDEmethods,Multiplebumps-Laing,ExploitingtheHamiltonian} and traveling waves in \cite{CoombesSchmidt,ExploitingtheHamiltonian}. We will study this kernel type in more detail in \cref{sec: Example_K2} by letting $a$ and $\theta$ vary.
\end{example}
%Example 3
\begin{example}
The third example is a kernel in class $\mathcal{C}_{\infty,\infty}$: 
\begin{equation}
K_3(x) := A \e{-a|x|}(c-\cos(bx)),
\end{equation}
where $a = 0.2$, $b = 2$, $c = 0.4$.
\end{example}
See Figures \ref{fig:k1_threshold}, \ref{fig:k2_threshold}, and \ref{fig:k3_threshold}.
\\ \\
{\bf Not pictured:} Kernels $K_1$, $K_2$, and $K_3$ satisfy wave speed conditions $(A_2)$, $(B_2)$, and $(C_1)$ respectively.

%FIGURE: K1_THRESHOLD
\begin{figure}[H]
\centering
\includegraphics[width=130mm]{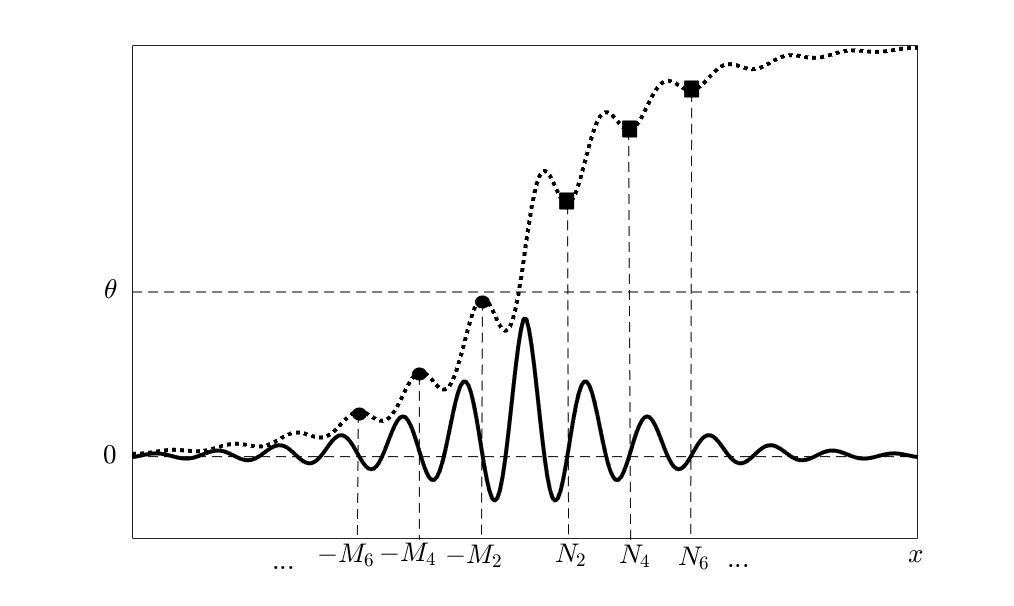}
\caption{Plot of $K_1(x)$ (solid) and $\displaystyle \f{\alpha}{2} - \alpha\integ{y}{x}[0]<K_1(y)>$ (dotted). The function $K_1$ satisfies $\mathcal{L}_\infty$ on the left half plane since the values at the dots are below the threshold and $\mathcal{R}_\infty$ on the right half plane since the values at the squares are above the threshold.}
\label{fig:k1_threshold}
\end{figure} \rule{0ex}{0ex}

%FIGURE: K2_THRESHOLD
\begin{figure}[H]
\centering
\includegraphics[width=130mm]{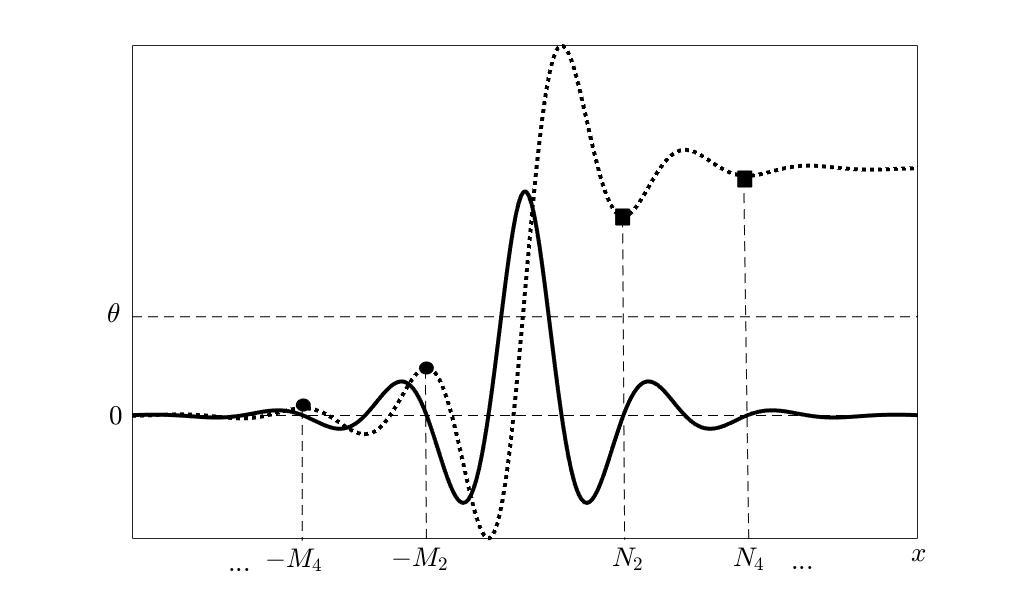}
\caption{Plot of $K_2(x)$ (solid) and $\displaystyle \f{\alpha}{2} - \alpha\integ{y}{x}[0]<K_2(y)>$ (dotted). The function $K_2$ satisfies $\mathcal{L}_\infty$ on the left half plane since the values at the dots are below the threshold and $\mathcal{R}_\infty$ on the right half plane since the values at the squares are above the threshold.}
\label{fig:k2_threshold}
\end{figure} \rule{0ex}{0ex}

%FIGURE: K3_THRESHOLD
\begin{figure}[H]
\centering
\includegraphics[width=130mm]{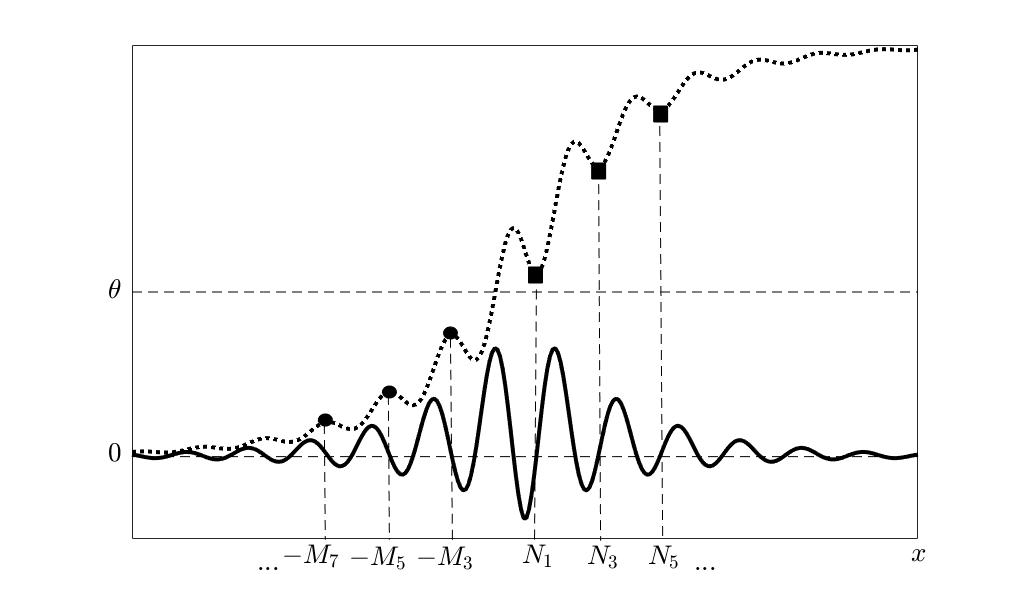}
\caption{Plot of $K_3(x)$ (solid) and $\displaystyle \f{\alpha}{2} - \alpha\integ{y}{x}[0]<K_3(y)>$ (dotted). The function $K_3$ satisfies $\mathcal{L}_{-\infty}$ on the left half plane since the values at the dots are below the threshold and $\mathcal{R}_{-\infty}$ on the right half plane since the values at the squares are above the threshold.}
\label{fig:k3_threshold}
\end{figure} %\rule{0ex}{0ex}

%subsection
\subsection{Previous Kernel Classes}
\label{subsec: previous_kernels}
\subsubsection*{Homogeneous Kernels}
For pure excitation kernels, existence, uniqueness, and stability of traveling wave solutions to \eqref{eq:paper_main} and similar models (seen as class $\mathcal{A}_{0,0}$) have been studied in tremendous depth; the resulting traveling wave solutions are monotonic with Evans functions that are easy to handle. 

The analysis immediately becomes more challenging when $K$ may become even a lateral inhibition or lateral excitation kernel. In such circumstances, Zhang \cite{Zhang-HowDo} first studied the existence, uniqueness, and stability of traveling wave solutions to \eqref{eq:paper_main}. For lateral inhibition kernels, he imposed the condition $\int_{-\infty}^0 |x|K(x) \, \mathrm{d}x \geq 0$ in order to guarantee uniqueness of wave speed. In this paper, such kernels can be regarded as lateral inhibition kernels that satisfy the wave speed condition $(A_1)$ and are therefore in class $\mathcal{A}_{1,1}$. The consideration that $\int_{-\infty}^0 |x|K(x) \, \mathrm{d}x < 0$  is necessarily the case where $K$ is in class $\mathcal{B}_{1,1}$, which is new. 

Other works (\cite{LijZhangSolo,magpantay2010wave})  have furthered the study by considering multiple delays such as 
\begin{equation}\label{eq: two_delays}
u_t + u = \alpha \integ{y}{\R}[]<K(x-y)H(u(y,t-\f{c_0}|x-y|)-\theta)> + \beta \integ{y}{\R}[]<J(x-y)H(u(y,t-\tau)-\theta)>,
\end{equation}
but the mechanics of the speed index function are similar to \cite{Zhang-HowDo}.
%Similarly, by noting that lateral excitation kernels are in class $\mathcal{C}_{1,1}$, we see that classes $\mathcal{A}_{j,k}$, $\mathcal{B}_{j,k}$ and $\mathcal{C}_{j,k}$ include and  generalize lateral inhibition and lateral excitation kernels.

Beyond lateral inhibition and lateral excitation kernels, others have explored kernels with more oscillations. Notably, Lv and Wang \cite{LvWang} first created five oscillatory kernel classes and proved the existence and uniqueness of the corresponding front solutions to \eqref{eq:paper_main}. Although their methods are important, all of their kernels cross the $x$-axis at most four times. Using similar techniques and kernel assumptions, Zhang et al. \cite{LijunZhangetal} improved their results by increasing the maximum number of oscillations to any finite number.

However, the main limitations in \cite{LvWang,LijunZhangetal} arise from the kernel assumptions they employ to guarantee the fronts have unique wave speeds. In particular, for kernels with more than one oscillation on the left half plane, they do not allow their speed index functions to have any critical points. Rather than using the repeated integral function $\Lambda^nK (x)$ and wave speed conditions from  \cref{sec: unique_assumptions}, they impose stronger conditions. In the most nonrestrictive cases, their kernels satisfy condition $(A_n)$ in the special case where $n=1$. Moreover, conditions $(B_n)$ and $(C_n)$ are not discussed at all. From a biological perspective, this means that oscillatory kernels modeling local inhibition were neglected.

On the other end of the spectrum, in a more general model,
Zhang and Hutt \cite{HuttZhang-TravelingWave} considered the unique wave speed problem for kernels that oscillated finitely or infinitely many times, but for condition ($A_n$) only. Moreover, they only studied the threshold requirements of fronts corresponding with pure excitation, lateral inhibition, and lateral excitation kernels. 
\subsubsection*{Heterogeneities}
Real neural tissue in the cortex certainly contains heterogeneities that if significant, may disrupt propagation. Such phenomena was explored by Bressloff \cite{bressloff2001traveling} and later in other studies \cite{avitabile2015snakes, coombes2011pulsating, kilpatrick2008traveling, schmidt2009wave}. Very briefly, we will highlight the main physical reasoning.

Consider model \eqref{eq:paper_main} without delay and with nonnegative kernels of the from $K(x,y)=\overline{K}(|x-y|)(1+a h\left(\f{y}{\epsilon}\right))$, where $a |h|\leq 1$, $h$ is periodic, and $0<\epsilon \ll 1$. As $\epsilon \to 0$, we recover the homogeneous case and existence and unqueness results from \cite{ExistenceandUniqueness-ErmMcLeod} may be applied. Using a perturbation argument and the spatial averaging theory, it is shown in \cite{bressloff2001traveling} that wave propagation failure occurs if $\epsilon$ or $a$ is too large. The intuition is that heterogeneities with larger amplitude and slower frequency cause a breakdown of propagation. 

We encourage the reader to investigate such studies and the references therein, as they draw analogies to the present work concerning the role of inhibition in preventing the existence of traveling fronts.
%SECTION: EU
\section{Existence and Uniqueness} \label{sec: analysis_Existence}
In this section, our goal is to rigorously prove \cref{thm: E_and_U}.

\subsection{Step 1: Formal solution} \label{subsec: step1}

Starting with the original scalar equation \eqref{eq:paper_main}, we let $z = x + \mu_0 t$ and $u(x,t) = U(z)$. Equation \eqref{eq:paper_main} immediately reduces to
\begin{equation} \label{eq_main_reduced_cov}
\mu_0 U' + U = \alpha \integ{y}{\R}[]<K(z-y)H\left(U\left(y -\f{\mu_0}{c_0}|z-y|\right)-\theta \right)>.
\end{equation}
Under the change of variable $\eta = y -\f{\mu_0}{c_0}|z-y|$, we have $z - y = \f{c_0}{c_0+sgn(z-\eta)\mu_0}(z-\eta)$ and \\ $\mathrm{d}y = \f{c_0}{c_0+sgn(z-\eta)\mu_0}\mathrm{d}\eta$. Using the assumption $U(\cdot) > \theta$ on $(0,\infty)$ and $U(\cdot) < \theta$ on $(-\infty,0)$, we arrive at the equation
\begin{equation} \label{eq:main_variationofpar}
\begin{split}
\mu_0 U' + U &= \alpha \integ{\eta}{0}[\infty]<\f{c_0}{c_0+sgn(z-\eta)\mu_0}\K{\f{c_0}{c_0+sgn(z-\eta)\mu_0}(z-\eta)}> \\
&= \alpha \integ{x}{-\infty}[\f{c_0z}{c_0+sgn(z)\mu_0}]<\K{x}>.
\end{split}
\end{equation}
Keeping the boundary conditions in mind, equation \eqref{eq:main_variationofpar} can easily be easily solved using variation of parameters. After solving and simplifying with integration by parts, we obtain the solution representation
\begin{equation} \label{eq:front_formal}
\begin{split}
U(z) &= \alpha \integ{x}{-\infty}[\f{c_0z}{c_0+sgn(z)\mu_0}]<K(x)> \\
&- \alpha \integ{x}{-\infty}[z]<\f{c_0}{c_0+sgn(x)\mu_0} \e{\f{x-z}{\mu_0}}\K{\f{c_0x}{c_0+sgn(x)\mu_0}}>,
\end{split}
\end{equation}
\begin{equation} \label{eq:frontderiv_formal}
U'(z) = \f{\alpha}{\mu_0}\integ{x}{-\infty}[z]<\f{c_0}{c_0+sgn(x)\mu_0} \e{\f{x-z}{\mu_0}}\K{\f{c_0x}{c_0+sgn(x)\mu_0}}>.
\end{equation}

\subsection{Step 2: Existence and Uniqueness of Wave Speed}
\label{subsec: wave_speed}
\subsubsection*{Existence}
Setting $U(0) = \theta$, as deemed necessary by Theorem \ref{thm: E_and_U}, we see that if $\mu_0 \in (0,c_0)$ exists and is unique, it must be the only solution to the equation
\begin{equation} \tag{\ref{eq:step2}}
\phi(\mu) = \f{2} - \f{\theta}{\alpha},
\end{equation}
where we recall from the introduction,
\begin{equation}
\tag{\ref{eq: speed_index}}
\phi(\mu) := \integ{x}{-\infty}[0]<\e{\f{c_0-\mu}{c_0\mu}x}\K{x}>
\end{equation}
is the speed index function. Since the integrand of $\phi$ is exponentially bounded, we use dominated convergence theorem and see that 
\begin{displaymath}
\lim_{\mu\to 0^{+}}\phi(\mu) = 0, \qquad \lim_{\mu \to c_0^{-}}\phi(\mu) = \integ{x}{-\infty}[0]<K(x)> = \f{2}.
\end{displaymath}
Finally, since $0<\f{2} - \f{\theta}{\alpha} <\f{2}$ by assumption, we use the intermediate value theorem to conclude there exists at least one solution to \eqref{eq:step2}. The main difficulty is utilizing the assumptions on the kernel to prove uniqueness. 

\subsubsection*{Uniqueness}
In this subsection, we prove \cref{lemma: unique_speed}. Part (i) was first proven in \cite[Subsection 2.3, parts d1,d2 p.~32-34]{HuttZhang-TravelingWave}\footnote{An inconsequential error was made that is corrected here.} for a similar model.

\begin{lemma}[Unique Wave Speed]
\label{lemma: unique_speed}
\leavevmode
\begin{itemize}
\item[(i)] Suppose $K$ satisfies wave speed condition ($A_n$) for some $n> 0$ and all other assumptions hold. Then $\phi'(\mu) > 0$ for all $\mu \in (0,c_0)$. Thus, there exists a unique solution $\mu_0 \in (0,c_0)$ to \eqref{eq:step2}.
\item[(ii)] Suppose $K$ satisfies wave speed condition ($B_n$) for some $n> 0$ and all other assumptions hold. Then $\phi$ has one critical point, a local maximum. Such a local maximum occurs when $\phi > \f{2}$ and thus, there exists a unique solution $\mu_0 \in (0,c_0)$ to \eqref{eq:step2}.
\item[(iii)] Suppose $K$ satisfies wave speed condition ($C_n$) for some $n>0$ and all other assumptions hold. Then $\phi$ has at most one critical point, a local minimum. Such a local minimum occurs when $\phi < 0$ and thus, there exists a unique solution $\mu_0 \in (0,c_0)$ to \eqref{eq:step2}.
\end{itemize}
\end{lemma}

\begin{proof}
\begin{itemize}

\item[(i)] 
Differentiating $\phi$ and using integration by parts $n$ times, we may write
\begin{align*}
\phi ' (\mu) & = \f{\mu^2} \integ{x}{-\infty}[0]<|x| \e{\f{c_0-\mu}{c_0\mu}x}K(x)>\\
& = \f{\mu^2} \integ{x}{-\infty}[0]< \e{\f{c_0-\mu}{c_0\mu}x}\Lambda^0 K(x)>\\
& = \f{\mu^2} \integ{x}{-\infty}[0]< \e{\f{c_0-\mu}{c_0\mu}x}\left[-\Lambda^1 K(x)\right]'>\\
& = \f{c_0-\mu}{c_0\mu}\f{\mu^2} \integ{x}{-\infty}[0]<\e{\f{c_0-\mu}{c_0\mu}x}\Lambda^1 K(x)> \\
\vdots \\
& = \left(\f{c_0-\mu}{c_0\mu}\right)^n\f{\mu^2} \integ{x}{-\infty}[0]<\e{\f{c_0-\mu}{c_0\mu}x}\Lambda^n K(x)> \label{eq: phi_deriv}. \numberthis
\end{align*}
Since $K$ satisfies wave speed condition $(A_n)$, we have $\Lambda^n K(x) \geq 0$ for $x\leq 0$.
Therefore, $\phi'(\mu)>0$ for all $\mu \in (0,c_0)$ so there exists a unique solution $\mu_0 \in (0,c_0)$ to \eqref{eq:step2}.

\item[(ii)] For all $n$, define the function
\begin{equation}
\label{eq: zeta_n}
\zeta_n(\mu) :=  \integ{x}{-\infty}[0]<\e{\f{c_0-\mu}{c_0\mu}x}\Lambda^n K(x)>.
\end{equation}
We observe that the sign of $\phi'$ and $\zeta_n$ are equivalent since $\mu < c_0$ and by \eqref{eq: phi_deriv},
\begin{equation}
\phi ' (\mu) = \left(\f{c_0-\mu}{c_0\mu}\right)^n \f{\mu^2} \zeta_n(\mu).
\end{equation}

Since $K$ satisfies wave speed condition $(B_n)$, there exists a constant $B_n > 0$ such that $\Lambda^n K(x) \geq 0 $ on $(-B_n,0)$ and $\Lambda^n K(x) \leq 0 $ on $(-\infty,-B_n)$. Suppose $\phi ' (\mu_*) = 0$ for some $\mu_*\in (0,c_0)$. This implies $\zeta_n(\mu_*) = 0$ and 
\begin{align*}
\zeta_n'(\mu_*) & = \f{\mu_*^2}\integ{x}{-\infty}[0]<|x|\e{\f{c_0-\mu_*}{c_0\mu_*}x}\Lambda^n K(x)> \\
& = \f{\mu_*^2}\left[\integ{x}{-B_n}[0]<|x|\e{\f{c_0-\mu_*}{c_0\mu_*}x}\Lambda^n K(x)> + \integ{x}{-\infty}[-B_n]<|x|\e{\f{c_0-\mu_*}{c_0\mu_*}x}\Lambda^n K(x)>\right] \\
& < \f{B_n}{\mu_*^2}\left[\integ{x}{-B_n}[0]<\e{\f{c_0-\mu_*}{c_0\mu_*}x}\Lambda^n K(x)> + \integ{x}{-\infty}[-B_n]<\e{\f{c_0-\mu_*}{c_0\mu_*}x}\Lambda^n K(x)>\right] \\
& = \f{B_n}{\mu_*^2}\zeta_n(\mu_*)\\
& = 0.
\end{align*}
Hence, $\zeta_n(\mu)$ changes signs from positive to negative at $\mu=\mu_*$. Since $\zeta_n$ and $\phi'$ have equivalent signs, we may conclude that $\phi$ has a local maximum at $\mu=\mu_*$ by the first derivative test. Moreover, since $\mu_*$ is arbitrary and $\phi$ is differentiable on $(0,c_0)$, we conclude that all critical points of $\phi$ must be local maximums. Finally, we have 
\begin{equation}
\label{eq: phi_lims}
\lim_{\mu\to 0^{+}}\phi(\mu) = 0, \qquad \lim_{\mu \to c_0^{-}}\phi(\mu) = \f{2}, \qquad 
\end{equation}
and
\begin{equation}\label{eq: phiprime_lims}
\lim_{\mu \to c_0^-} \phi'(\mu) = \f{c_0^{2n}}\lim_{\mu \to c_0^-} (c_0-\mu)^n \integ{x}{-\infty}[0]<\e{\f{c_0-\mu}{c_0\mu}x}\Lambda^n K(x)>
\end{equation}
so if $\Lambda^n K(-\infty)\in [-\infty,0)$, the limit in \eqref{eq: phiprime_lims} will be negative and such a $\mu_*$ will exist by the intermediate value theorem with $\phi(\mu_*)>\f{2}$. Therefore, $\phi$ is strictly increasing when $0<\phi<\f{2}$ so there exists a unique solution $\mu_0\in (0,c_0)$ to \eqref{eq:step2}.

\item[(iii)]  We proceed similarly to (ii). Since $K$ satisfies wave speed condition $(C_n)$, there exists a constant $C_n > 0$ such that $\Lambda^n K(x) \leq 0 $ on $(-C_n,0)$ and $\Lambda^n K(x) \geq 0 $ on $(-\infty,-C_n)$. Suppose $\phi'(\mu_*) = 0$ for some $\mu_* \in (0,c_0)$. Then as before, $\zeta_n(\mu_*) = 0$ and 
\begin{align*}
\zeta_n'(\mu_*) & = \f{\mu_*^2}\integ{x}{-\infty}[0]<|x|\e{\f{c_0-\mu_*}{c_0\mu_*}x}\Lambda^n K(x)> \\
& = \f{\mu_*^2}\left[\integ{x}{-C_n}[0]<|x|\e{\f{c_0-\mu_*}{c_0\mu_*}x}\Lambda^n K(x)> + \integ{x}{-\infty}[-C_n]<|x|\e{\f{c_0-\mu_*}{c_0\mu_*}x}\Lambda^n K(x)>\right] \\
& > \f{C_n}{\mu_*^2}\left[\integ{x}{-C_n}[0]<\e{\f{c_0-\mu_*}{c_0\mu_*}x}\Lambda^n K(x)> + \integ{x}{-\infty}[-C_n]<\e{\f{c_0-\mu_*}{c_0\mu_*}x}\Lambda^n K(x)>\right] \\
& = \f{C_n}{\mu_*^2}\zeta_n(\mu_*)\\
& = 0.
\end{align*}
We apply the first derivative test on $\phi$ again; this time we find that $\phi$ has a local minimum at $\mu = \mu_*$ and conclude that critical points of $\phi$ must be local minimums with $\phi(\mu_*)<0$. Therefore, by similar reasoning to (ii), $\phi$ is strictly increasing when $0<\phi<\f{2}$ so there exists a unique solution $\mu_0\in (0,c_0)$ to \eqref{eq:step2}.

% \begin{equation*} 
% \integ{x}{-\infty}[0]<\e{\f{x}{a_*}}\Lambda^n K(x)> = 0
% \end{equation*}
% and
% \begin{align*}
% \phi ''(\mu_*) & =\f{a_*^{n}\mu_*^4}\left[\integ{x}{-C_n}[0]< |x| \e{\f{x}{a_*}}\Lambda^n K(x)> + \integ{x}{-\infty}[-C_n]< |x| \e{\f{x}{a_*}}\Lambda^n K(x)>\right] \tag{\ref{eq:sec_deriv_part}} \\
% & > \f{C_n}{a_*^n\mu_*^4}\integ{x}{-\infty}[0]< \e{\f{x}{a_*}}\Lambda^n K(x)> \\
% & = 0. 
% \end{align*}
% Again, we find by the second derivative test that $\phi$ has a local minimum at $\mu = \mu_*$ and conclude that critical points of $\phi$ must be local minimums. By a change of variable, we may write
% \begin{equation}
% \phi'(\mu) = \int_{-\infty}^0 |x| \e{\f{c_0-\mu}{c_0}x}K(\mu x) \, \mathrm{d}x.
% \end{equation}
% A trivial calculation shows that the sign of $\phi'(0)$ is determined by the sign of $K(0)$. Finally, since $K$ satisfies wave speed condition $(C_n)$, it follows that $K(0)\leq 0$; if $K(0)<0$, such a $\mu_*$ will exist with $\phi(\mu_*)<0$. Therefore, by similar reasoning to (ii), $\phi$ is strictly increasing when $0<\phi<\f{2}$ so there exists a unique solution $\mu_0\in (0,c_0)$ to \eqref{eq:step2}.
\end{itemize}
\end{proof}
See Figure \ref{fig:speed_index}. As a corollary to \cref{lemma: unique_speed}, we may easily answer an open problem: does a front solution to \eqref{eq:paper_main} arising from a lateral inhibition kernel with $\int_{-\infty}^0 |x|K(x)\,\mathrm{d}x<0$ have a unique wave speed?
\begin{corollary}
\label{cor: lat_unique_speed}
Suppose $K$ is any lateral inhibition kernel. Then there exists a unique solution $\mu_0 \in (0,c_0)$ to \eqref{eq:step2} independent of the value of $\int_{-\infty}^0 |x|K(x)\,\mathrm{d}x$.
\end{corollary}
\begin{proof}
If $\int_{-\infty}^0 |x|K(x)\,\mathrm{d}x\geq 0$, then $K$ is in class $\mathcal{A}_{1,1}$. Otherwise, $K$ is in $\mathcal{B}_{1,1}$. Applying \cref{lemma: unique_speed} (i)-(ii), the result follows.
\end{proof}

Combining the results of Step 1 with \cref{lemma: unique_speed}, we see that if our formal solutions are in fact real fronts, they will have unique wave speeds. The last step we must take is proving that the solutions satisfy the necessary threshold requirements to be actual solutions. After this step, our proof of existence and uniqueness of front solutions to \eqref{eq:paper_main} will be complete.

%FIGURE: speed index
\begin{figure}[H]
\centering
\includegraphics[width=130mm]{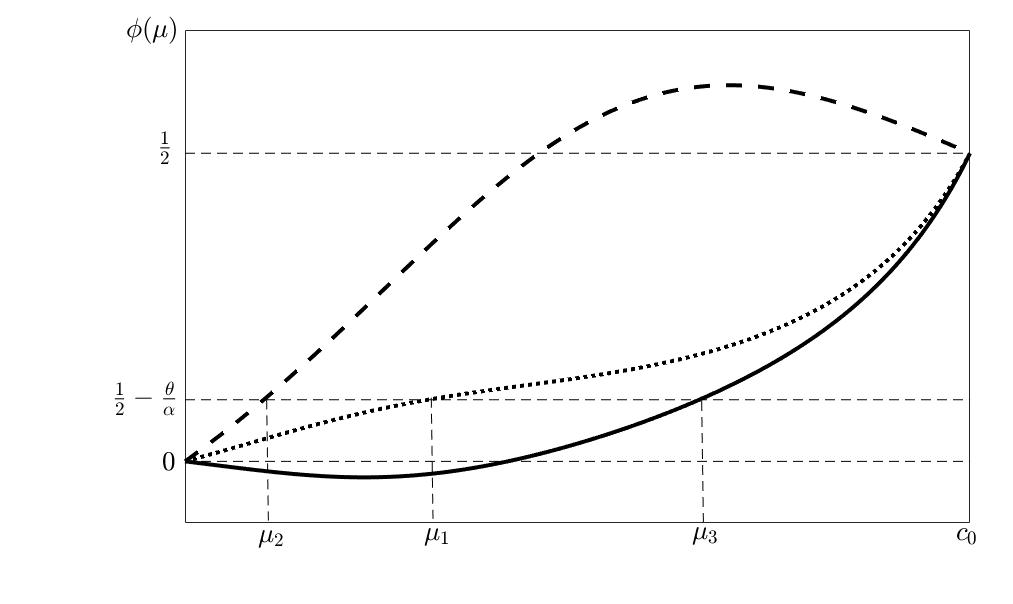}
\caption{Plots of $\phi(\mu)$ for kernels $K_1$ (dotted), $K_2$ (dashed), and $K_3$ (solid).}
\label{fig:speed_index}
\end{figure} \rule{0ex}{0ex}

\subsection{Step 3: Formal Solution is a Real Solution} \label{subsec: step3}
The main intent of this step is to use the threshold conditions outlined in \cref{subsubsec: left_and_right} to prove that $U(0) = \theta$, $U(z) < \theta$ on $(-\infty,0)$, and $U(z) > \theta$ on $(0,\infty)$. We analyze the left and right half planes separately.  

Following the approach by others, our technique is to show that on the left half plane, our conditions on $K$ are sufficient in guaranteeing that all possible local maximums of $U$ lie below the threshold. On the right half plane, we show that all possible local minimums of $U$ lie above the threshold.

We now prove the following lemmas using a modified version of \cite[Lemma 3 and 4, p.~6]{LijunZhangetal}. The modifications account for the case where $K$ oscillates infinitely many times and when $K$ satisfies $\mathcal{L}_{-j}$ for $j > 1$.
\begin{lemma} \label{lemma: U_less_theta}
If $K$ satisfies $\mathcal{L}_{j}$ or $\mathcal{L}_{-j}$ for some $j\geq 0$, then $U(z)<\theta$ on $(-\infty,0)$.
\end{lemma}

\begin{proof}
If $K$ satisfies $\mathcal{L}_j$ for some $0\leq j < \infty$, see \cite[Lemma 4, p.~6]{LijunZhangetal}. The case where $j = \infty$ is a trivial extension of the same argument: since $\alpha \int_{-\infty}^0 K(x) \, \mathrm{d}x = \f{\alpha}{2}$, there exists a positive integer $N(\alpha,\theta)$ such that $\f{\alpha}{2} - \alpha \int_{-M_{2n}}^0 K(x) \, \mathrm{d}x < \theta$ for all $n\geq N(\alpha,\theta)$. Hence, the argument for the $0\leq j < \infty$ case can be applied again.

Suppose $K$ satisfies $\mathcal{L}_{-j}$ for some $j\geq 1$. For $z\leq 0$, define 
\begin{equation}
\begin{split}
\psi(z) := & \integ{x}{-\infty}[z]<\e{\f{c_0-\mu_0}{c_0\mu_0}x}K(x)>\\
= & \left(\f{2}-\f{\theta}{\alpha}\right) - \integ{x}{z}[0]<\e{\f{c_0-\mu_0}{c_0\mu_0}x}K(x)>.
\end{split}
\end{equation}

The sign of $U'(\f{c_0-\mu_0}{c_0}z)$ is determined by $\psi(z)$ since

\begin{equation}
\begin{split}
U'\left(\f{c_0-\mu_0}{c_0}z\right)& =
\f{\alpha}{\mu_0}\e{-\f{c_0-\mu_0}{c_0\mu_0}z}\integ{x}{-\infty}[z]<\e{\f{c_0-\mu_0}{c_0\mu_0}x}K(x)>\\
& =\f{\alpha}{\mu_0}\e{-\f{c_0-\mu_0}{c_0\mu_0}z}\psi(z).
\end{split}
\end{equation}

If $j = 1$, then $\psi(z)$ is increasing on $(-\infty,-M_1)$, decreasing on $(-M_1,0)$ with $\psi(0) = \f{2}-\f{\theta}{\alpha}>0$ and $\underset{z\to -\infty}{\lim}\psi(z) = 0$. Therefore, $\psi(z)\geq 0$ on $(-\infty,0)$ so $U(z)$ is monotonic on $(-\infty,0)$.

If $j = 2$, then $\psi(z)$ is decreasing on $(-\infty,-M_2) \cup (-M_1,0)$, increasing on $(-M_2,-M_1)$. Since $\psi(0) > 0$ and $\underset{z\to -\infty}{\lim}\psi(z) = 0$, we conclude that $\psi(z)$ changes signs exactly one time, from negative to positive for some $z_*\in (-M_2,-M_1)$. But this mean $U(z)$ has a local minimum at $z = \f{c_0-\mu_0}{c_0}z_*$ and no local maximums on $(-\infty,0)$.

If $j\geq 3$, then $\psi(z)$ is increasing on $(-M_{2n+2},-M_{2n+1})$ and decreasing on $(-M_{2n+1},-M_{2n})$ for $n\geq 0$, where we let $M_0=0$. Based again on the fact that $\psi(0)>0$ and $\underset{z\to -\infty}{\lim}\psi(z) = 0$, we see that if $\psi(z)$ changes signs from positive to negative at $z = z_*$, and therefore $U(z)$ has a local maximum at $z = \f{c_0-\mu_0}{c_0}z_*$, then $z_*\in (-M_{2n+1},-M_{2n})$ for $n\geq 1$. But then
\begin{equation}
\begin{split}
U\left(\f{c_0-\mu_0}{c_0}z_*\right) & = \alpha \integ{x}{-\infty}[z_*]<K(x)> - \mu_0 U ' \left(\f{c_0-\mu_0}{c_0}z_*\right) \\
& = \alpha \integ{x}{-\infty}[z_*]<K(x)> \\
& = \f{\alpha}{2}-\alpha \integ{x}{z_*}[0]<K(x)> \\
& < \f{\alpha}{2}-\alpha \integ{x}{-M_{2n+1}}[0]<K(x)>\\
& < \theta.
\end{split}
\end{equation}
Therefore, $U$ stays below the threshold on all possible local maximums.
\end{proof}

\begin{lemma} \label{lemma: U_greater_theta}
If $K$ satisfies $\mathcal{R}_{k}$ or $\mathcal{R}_{-k}$ for some $k\geq 0$, then $U(z)>\theta$ on $(0,\infty)$.
\end{lemma}
\begin{proof}
If $k<\infty$, see \cite{LijunZhangetal}. The extension to the $k \to \infty$ case is similar to the $j \to \infty$ case in \cref{lemma: U_less_theta}.
\end{proof}

\subsection{Summary of Existence and Uniqueness}
Combining the arguments \cref{subsec: step1,subsec: wave_speed,subsec: step3}, we have proved \cref{thm: E_and_U}.
\begin{corollary}\label{cor: EU_MH}
Suppose that $0<2\theta<\alpha$ and $K$ is any lateral inhibition kernel. Then independent of the value of $\int_{-\infty}^0 |x|K(x)\,\mathrm{d}x$, there exists a unique traveling wave front to \eqref{eq:paper_main} as described by \cref{thm: E_and_U}.
\end{corollary}
\begin{proof}
For lateral inhibition kernels, it is well known (see \cite{Zhang-HowDo}) that the argument for existence of at least one front solution to \eqref{eq:paper_main} is independent of the value  of $\int_{-\infty}^0 |x|K(x) \, \mathrm{d}x$. By \cref{cor: lat_unique_speed}, uniqueness of wave speed is established.
\end{proof}

See Figure \ref{fig:wave_fronts} for plots of $U(z)$ with kernels $K_1$, $K_2$, and $K_3$.
%FIGURE: wave_fronts
\begin{figure}[H]
\centering
\includegraphics[width=130mm]{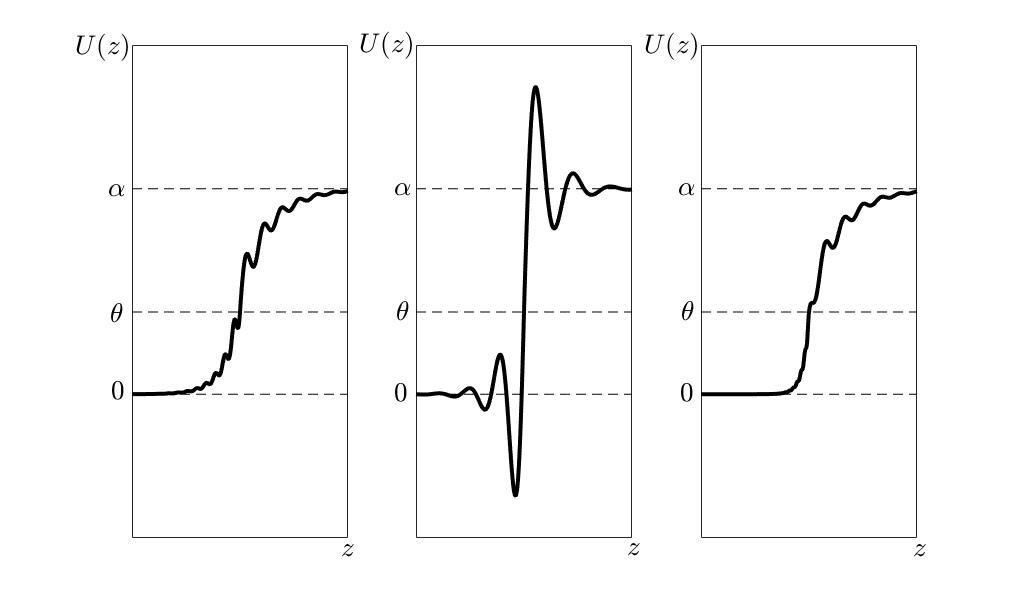}
\caption{Plots of $U(z)$ for kernels $K_1$ (left), $K_2$ (center), and $K_3$ (right).}
\label{fig:wave_fronts}
\end{figure}% \rule{0ex}{0ex}

%SECTION: STABILITY
\section{Spectral Stability}
\label{sec: spectral_stability}

%Originally derived by Zhang \cite{Zhang-OnStability}, we use the same Evans function as others have for our model -- namely, by translation invariance of the front,
%\begin{equation} \label{eq:Evans_main}
%\mathcal{E}(\lambda) := 1-\f{\phi(\f{\mu_0}{\lambda + 1})}{\phi(\mu_0)},
%\end{equation}
%with domain
%\begin{equation*}
%\Omega = \{\lambda\in\C \mid \text{Re}(\lambda)\ > -1\}.
%\end{equation*}
%We must show that $\mathcal{E}(\lambda)$ has no roots on the right half plane except for an algebraically simple root at $\lambda = 0$. Using surface plots of $|\mathcal{E}(\lambda)|$, we show that several examples of previously unstudied kernels lead to spectral stability. Finally, we prove \cref{thm: stability} and explore some of the kernel types
A natural concern is whether the solutions satisfy any criteria for stability. We have shown that a physiologically motivated model permits traveling wave solutions. However, actual neuronal networks are surely more complex than our model. Hence, in order for our analysis to be biologically useful, our solutions must ideally persist under the influence of small perturbations.

Motivated by the pioneering work of Evans \cite{Evans-I,Evans-II,Evans-III,Evans-IV} on the Hodgkin-Huxley model and Jones \cite{Jones-StabilityoftheTravelingWave} on the Fitzhugh-Nagumo model, stability of traveling wave solutions has also been well developed for nonlocal equations. In particular, Zhang \cite{Zhang-OnStability} studied spectral stability of \eqref{eq: pulse1}-\eqref{eq: pulse2} in the case where $c_0 = \infty$ (no delay); he formulated a closed form Evans function and used it to solve the linear eigenvalue problem. The problem came full circle when Sandstede \cite{Sandstede-EvansFunctions} proved that in appropriate function spaces, spectral stability implies nonlinear stability. Interestingly, when there is delay, a rigorous proof of a similar connection has not been given and is still an open problem. The main difficulty is the fact that the eigenvalue problem becomes nonlinear in $\lambda$.

Some partial results have been made: Zhang \cite{Zhang-HowDo} studied the spectral stability of wave front solutions to \eqref{eq:paper_main} arising from pure excitation, lateral inhibition, and lateral excitation kernels. Coombes and Owen \cite{CoombesOwen_Evans} analyzed and discussed the spectral stability in a more general setting (such as in \eqref{eq: original_main}) using similar techniques to Zhang. We remark that stability was not considered at all in \cite{LijunZhangetal, LijZhangSolo, LvWang, magpantay2010wave}, some of the main references that motivate this work.

In an effort to view the front as a stationary solution, we change $x$ to the traveling coordinate $z$ and write \eqref{eq:paper_main} as
\begin{equation} \label{eq:stab_main}
P_t + \mu_0 P_z + P = \alpha \integ{y}{\R}<\K{z-y}H(P\left(y-\f{\mu_0}{c_0}|z-y|,t-\f{c_0}|z-y|\right)-\theta)>.
\end{equation}
Letting $p(z,t) = P(z,t)-U(z)$ and linearizing \eqref{eq:stab_main} about the wave front, we yield
\begin{equation} \label{eq: linearized_main}
p_t+\mu_0 p_z + p = \f{\alpha c_0}{U'(0)(c_0+sgn(z)\mu_0)}\K{\f{c_0}{c_0+sgn(z)\mu_0}z}p\OO{0}{t-\f{|z|}{c_0+sgn(z)\mu_0}}.
\end{equation}
Setting $p(z,t) = \e{\lambda t}\psi(z)$, we produce a nonlinear eigenvalue problem $\mathcal{L}(\lambda)\psi  = \lambda \psi$, where the family of operators $\mathcal{L}(\lambda):C^1(\R)\cap L^\infty(\R) \to C^0(\R)\cap L^\infty(\R)$ are defined by 
\begin{equation}
\label{eq: L_psi}
\begin{split}
\mathcal{L}(\lambda)\psi = & - \mu_0\psi' - \psi +\f{\alpha c_0}{U'(0)(c_0+sgn(z)\mu_0)}\K{\f{c_0}{c_0+sgn(z)\mu_0}z} \\
& \times \e{-\f{\lambda |z|}{c_0+sgn(z)\mu_0}} \psi(0).
\end{split}
\end{equation}
Using standard notation, we denote by $\sigma (\mathcal{L})$ the spectrum of $\mathcal{L}$, which is made up of the point spectrum and essential spectrum; the point spectrum is made up of eigenvalues.
We use the following important definition.
\begin{definition} \label{def: spec_stab}
A traveling wave solution to \eqref{eq:paper_main} is {\it spectrally stable} if the following conditions hold:
\begin{itemize}
\item[(i)] The essential spectrum $\sigma_{essential}(\mathcal{L})$ lies entirely to the left of the imaginary axis.
\item[(ii)] The eigenvalue $\lambda = 0$ is algebraically simple.
\item[(iii)] There exists a positive constant $\kappa_0 > 0$ such that $
\max\{\text{Re }\lambda \mid \lambda \in \sigma_{point}(\lambda),\lambda \neq 0\}\leq -\kappa_0$.
\end{itemize}
\end{definition}
\subsubsection*{Essential Spectrum}
The essential spectrum of $\mathcal{L}(\lambda)$ is easily determined by the intermediate eigenvalue problem \\ $\mathcal{L}^\infty \psi = \lambda \psi$, where
\begin{equation}
\mathcal{L}^\infty \psi := -\mu_0\psi' - \psi,
\end{equation}
which is linear in $\psi$ and $\lambda$. A trivial calculation shows the intermediate eigenvalue problem is solved by
\begin{equation}
\label{eq: sol_Linfinity}
\psi_0(\lambda,z) = C(\lambda)\e{-\f{\lambda+1}{\mu_0}z}.
\end{equation}
Assuming $C(\lambda) \neq 0$, the solution $\psi_0(\lambda,z)$ blows up as $z\to -\infty$ when $\text{Re}(\lambda) > -1$ and as $z\to \infty$ when $\text{Re}(\lambda) <-1$. Thus, the essential spectrum is the vertical line 
\begin{equation}
\begin{split}
\sigma_{essential} = \{\lambda \in \C \mid \text{Re}(\lambda) = -1 \},
\end{split}
\end{equation}
which safely stays entirely on the left half plane so \cref{def: spec_stab} (i) is satisfied. Hence, on the domain
\begin{equation*}
\Omega = \{\lambda\in\C \mid \text{Re}(\lambda)\ > -1\},
\end{equation*}
the stability is determined by $\sigma_{point}$, which notably depends on $K$. 

\subsubsection*{Point Spectrum}
Through a series of calculations \cite{CoombesOwen_Evans,Zhang-HowDo}, the eigenvalue problem is solved by
\begin{equation}
\label{eq: L_sol}
\begin{split}
\psi(\lambda,z) = & C(\lambda)\e{-\f{\lambda+1}{\mu_0}z} + \f{\alpha\psi(\lambda,0)}{\mu_0 U'(0)}\int_{-\infty}^z \f{c_0}{c_0+sgn(x)\mu_0}\e{\f{\lambda+1}{\mu_0}(x-z)}
\\
& \times \e{-\f{\lambda |x|}{c_0+sgn(x)\mu_0}}\K{\f{c_0x}{c_0+sgn(x)\mu_0}} \, \mathrm{d}x.
\end{split}
\end{equation}
Since $\psi(\lambda,0)$ appears on the right hand side in \eqref{eq: L_sol}, we must plug in $z = 0$ and solve for $C(\lambda)$ to ensure $\psi$ is well-defined. For $\lambda \in \Omega$, nontrivial solutions $\psi(\lambda,z)$ do not blow up as $z \to \pm \infty$ if and only if $C(\lambda) = 0$. We obtain an Evans function whose zeros entirely determine $\sigma_{point}$. By translation invariance, $\lambda = 0$ is an eigenvalue with eigenfunction $U'(z)$. Hence, the Evan's function, written in terms of the speed index function $\phi$, is given explicitly by
\begin{equation} \tag{\ref{eq:Evans_main}}
\mathcal{E}(\lambda) := 1-\f{\phi(\f{\mu_0}{\lambda + 1})}{\phi(\mu_0)}.
\end{equation}
Since the formulation of $\mathcal{E}$ is not a new result, we present the following lemma without proof. See \cite{Zhang-OnStability} for the details.
\begin{lemma}
\label{lemma: Evans_prop}
\leavevmode
\begin{itemize}
\item[(i)] The Evans function is complex analytic on $\Omega$ and real if $\lambda$ is real.
\item[(ii)] The complex number $\lambda_0 \in \sigma_{point}$ if and only if $\mathcal{E}(\lambda_0) = 0$.
\item[(iii)] The algebraic multiplicity of eigenvalues of $\mathcal{L}(\lambda)$ is exactly equal to the multiplicity of zeros of $\mathcal{E}(\lambda)$.
\item[(iv)] In the domain $\Omega$, the Evans function has the asymptotic behavior $\underset{|\lambda| \to \infty}{\lim} \mathcal{E}(\lambda) = 1$.
\end{itemize}
\end{lemma}

\subsubsection*{Point Spectrum $\lambda \in \R^+ \cup \{0\}$}
We already know $\lambda = 0$ is an eigenvalue of multiplicity at least one. Moreover,
\begin{equation}
\label{eq: Evans_deriv}
\mathcal{E}'(\lambda) = \f{\mu_0}{(\lambda + 1)^2}\f{\phi '(\f{\mu_0}{\lambda +1})}{\phi(\mu_0)},
\end{equation}
which implies $\mathcal{E}'(0) = \mu_0\f{\phi'(\mu_0)}{\phi(\mu_0)} > 0$ for all $K$ in classes $\mathcal{A}_{j,k}$, $\mathcal{B}_{j,k}$, and  $\mathcal{C}_{j,k}$. By \cref{lemma: Evans_prop} (iii), it follows that $\lambda = 0$ is a simple eigenvalue so \cref{def: spec_stab} (ii) is satisfied.

We now consider the behavior of $\mathcal{E}(\lambda)$ for $\lambda$ on the positive real axis. For such $\lambda$, by uniqueness of the wave speed, $\phi(\f{\mu_0}{\lambda + 1}) \neq \phi(\mu_0)$. Therefore, $\mathcal{E}(\lambda) \neq 0$ so $\lambda$ cannot be an eigenvalue.
\subsubsection*{Point Spectrum Re$(\lambda)\geq 0$, Im$(\lambda)\neq0$}

We are finally brought to the most difficult part of the stability analysis: if any, where are the zeros of $\mathcal{E}(\lambda)$ when $\text{Re}(\lambda) \geq 0$ and $\text{Im}(\lambda) \neq 0$? We recall from \cref{subsubsec: math_diff} that in its fullest generality, this problem is still open.

%As evident by the plots in Figure \ref{fig:Evans_mag}, we see that for at least some of our kernels, even the ones that oscillate countably many times, we have spectral stability of our wave fronts. However, it is currently unclear how often instability occurs, if at all. If counterexamples can be given, are there any defining features of the kernels that lead to instability?

In pursuit of a meaningful stability result, we now exploit uniqueness and prove \cref{thm: stability}. As evident from the analysis above, the following lemma completes the proof.

\begin{lemma}\label{lemma: point_spectrum}
Suppose that $0<2\theta<\alpha$, $K$ is in class $\mathcal{A}_{j,k}$, $\mathcal{B}_{j,k}$, or $\mathcal{C}_{j,k}$ for some integers $j$ and $k$, and $U$ is a unique solution described in \cref{thm: E_and_U}. If the Laplace transform of $K$ satisfies
\begin{equation}\label{eq: Laplace}
\integ{x}{-\infty}[0]<\e{s x}K(x)>=\f{p(s)}{q(s)},
\end{equation}
where $p$ and $q$ are polynomials of degree at most two, then $\mathcal{E}(\lambda)\neq 0$ when Re$(\lambda)\geq 0$, Im$(\lambda)\neq0$.
\end{lemma}
\begin{proof}
By the definition of $\mathcal{E}$, any roots satisfy $\phi\left(\f{\mu_0}{\lambda+1}\right)=\phi(\mu_0)=\f{2}-\f{\theta}{\alpha}$ so
$$
\f{2}-\f{\theta}{\alpha}= \f{p\left(\f{\lambda+1}{\mu_0}\right)}{q\left(\f{\lambda+1}{\mu_0}\right)},
$$
which means
$$
\overline{p}(\lambda):=\left(\f{2}-\f{\theta}{\alpha}\right) q\left(\f{\lambda+1}{\mu_0}\right)-p\left(\f{\lambda+1}{\mu_0}\right)
$$
is a polynomial of degree at most two and $\overline{p}(\lambda)=0$ if and only if $\mathcal{E}(\lambda)=0$. Therefore, $\overline{p}(0)=\mathcal{E}(0)=0$ so if $\overline{p}$ is of degree one, then we are done. Otherwise, by simplicity, $\overline{p}$ must have one other distinct real root, say $\lambda_*$. If $\lambda_* >0$, then $\phi\left(\f{\mu_0}{\lambda_*+1}\right)=\f{2}-\f{\theta}{\alpha}$, contradicting the uniqueness of $\mu_0$. Therefore, $\lambda_*<0$ and the claim follows.
\end{proof}
\subsubsection*{Proof of \cref{thm: stability}}
By combining the results in this section, we see that \cref{def: spec_stab} (i)-(iii) are satisfied and the proof of spectral stability is complete. Finally, if $c_0=\infty$, the result in  \cite{Sandstede-EvansFunctions} shows the equivalence of spectral, linear and nonlinear stability, completing the proof of \cref{thm: stability}.

\begin{corollary}\label{cor: sin_cos_stab}
 Suppose that $0<2\theta<\alpha$, $K$ is in class $\mathcal{A}_{j,k}$, $\mathcal{B}_{j,k}$, or $\mathcal{C}_{j,k}$ for some integers $j$ and $k$, and $U$ is a unique solution described in \cref{thm: E_and_U} with $K$ of the form
 $$
 K(x)= \e{-a|x|}(b\cos(cx+d)+e\sin(cx+f)).
 $$
Then $U$ is spectrally stable.
\end{corollary}
\begin{corollary}
Suppose that $0<2\theta<\alpha$ and $K(x)= \e{-a|x|}(-b|x|+c)$ with $a,\, b,\, c$ positive and $K$ normalized. Then there exists unique and stable front solutions to \cref{eq:paper_main}.
\end{corollary}

\subsubsection*{Computational Methods}
We close this section by pointing out that since $\mathcal{E}$ is complex analytic, there are a number of options at our disposal. By \cref{lemma: Evans_prop} (iv) and the fact that zeros of $\mathcal{E}$ must be isolated, the number of zeros on the right half plane is a finite number and contained in the region $$
B_{\delta,R} = \{\lambda \in \Omega \mid \text{Re}(\lambda)\geq0 \text{ and }\delta\leq|\lambda| \leq R\}
$$ 
for some $R\gg 0$, $\delta>0$. Common tools like maximum principle or argument principle can be applied.

For example, it is certainly true that $|1-\mathcal{E}(\lambda)| < 1$ when $|\lambda| = R$. If we can show $|1-\mathcal{E}(\lambda)| < 1$ along the imaginary axis when $\delta \leq|\lambda| \leq R$, then $|1-\mathcal{E}(\lambda)| < 1$ on $B_{\delta,R}$ by the maximum principle. Hence $|\mathcal{E}(\lambda)|>0$ on $B_{\delta,R}$ and stability follows. 

Finally, when specific kernels are given, we can use mathematical software to look at real and imaginary contour plots of $\mathcal{E}(\lambda)$ (as was done in \cite{CoombesOwen_Evans}) or surface plots of $|\mathcal{E}(\lambda)|$ on the right half plane. In the examples given in Figure \ref{fig:Evans_mag}, we have chosen the latter option.

%FIGURE: K3_THRESHOLD
\begin{figure}[H]
\centering
\includegraphics[width=130mm]{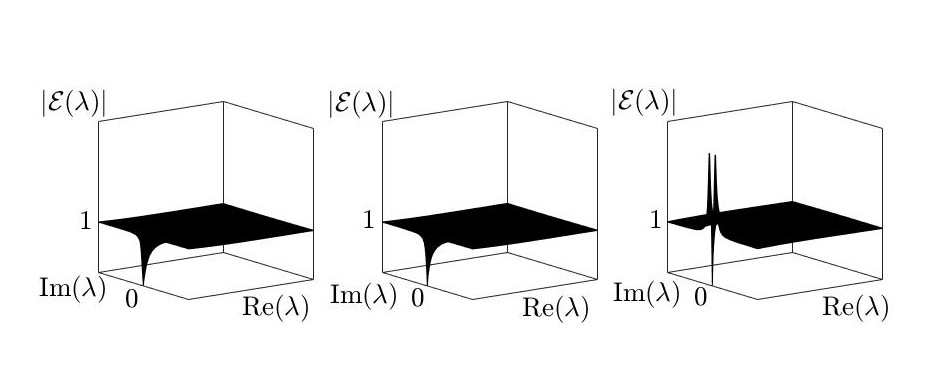}
\caption{Surface plot of $|\mathcal{E}(\lambda)|$ with kernels $K_1$ (left), $K_2$, (center), and $K_3$ (right) when Re$(\lambda)\geq 0$. In all cases, $|\mathcal{E}(\lambda)|>0$ except at $\lambda = 0$, showing spectral stability.}
\label{fig:Evans_mag}
\end{figure} \rule{0ex}{0ex}

\section{Classifying Example $K_2$ With Varying Thresholds} \label{sec: Example_K2}
In this section, we look further at the one parameter, symmetric kernel type discussed in \cite{ExploitingtheHamiltonian}:
\begin{equation}\label{eq: kernel_elvin}
K(x;a):= C(a)\e{ -a |x|}(a \sin(|x|) + \cos(x)),
\end{equation}
where $a>0$ and $C(a)$ is chosen as a normalizing constant. Our goal is to use the assumptions in \cref{sec: kernel_constructions} to fully classify existence, uniqueness, and stability of fronts by class. Throughout this section, assume $c_0=\infty$, $\alpha=1$, and $0<\theta<\f{2}$.
\subsection*{Wave Speed Conditions ($A_n$) and ($B_n$)} \label{subsec: Ex_AnBn}
Define
$$
f(a):=\integ{x}{\infty}[0]<|x|K(x;a)>.
$$
For fixed $a>0$, it is clear from the definition of wave speed conditions in \cref{sec: unique_assumptions}, the requirements (i): $f(a)\geq 0$ and (ii): $f(a)\leq 0$ are necessary in order for $K$ to satisfy ($A_n$) and ($B_m$) for some $n$ and $m$ respectively.

Regarding \eqref{eq: kernel_elvin}, we see from Figure \ref{fig: plot_f} that there exists a unique $a^*>0$ such that $f(a)\leq 0$ for $a\leq a^*$ and $f(a)\geq 0$ for $a\geq a^*.$  We omit the details, but it can be shown that for $a\neq a^*$, requirements (i) and (ii) are also sufficient. Indeed, $K$ satisfies ($A_2$) for $a\geq a^*$ and ($B_2$) for $a<a^*$.
\begin{figure}[H]
\centering
\includegraphics[width=130mm]{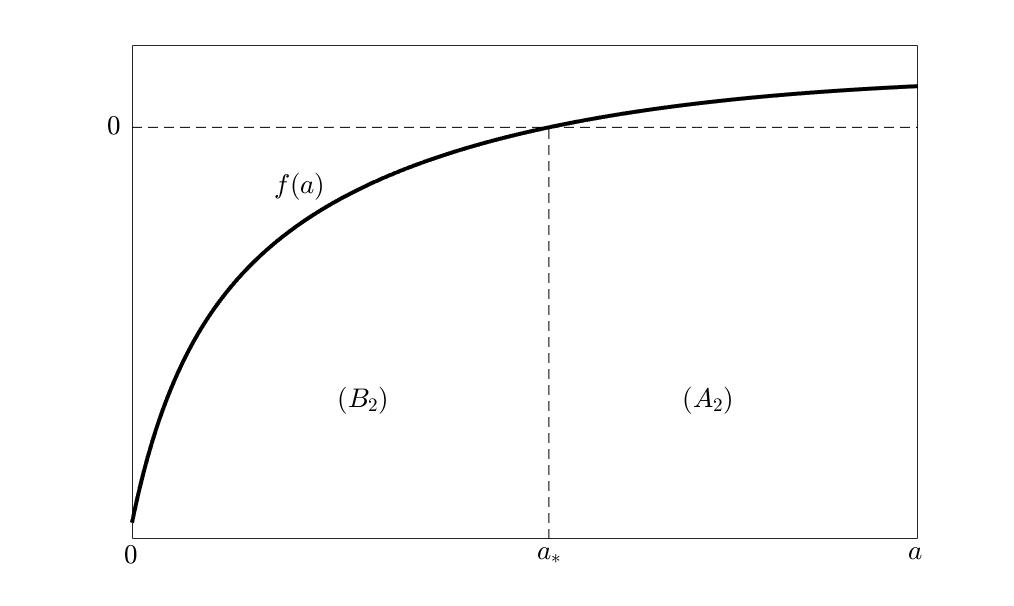}
\caption{Plot of $f(a)$. When $f(a_*)=0$, the kernel changes from satisfying wave speed condition ($B_2$) to ($A_2$). A numerical calculation shows $a_*=.5774$.}
\label{fig: plot_f}
\end{figure}
\subsection*{Unique Wave Speed} \label{subsec: ex_wavespeed}
We recall that the wave speed is determined by solutions to $\phi(\mu)=\f{2}-\theta$. By the definition of $\phi$ and using $K$ from \eqref{eq: kernel_elvin}, a series of calculations reveals $\f{\mu}$ is obtained as roots to the quadratic equation
\begin{equation}\label{eq: wave_quadratic}
c_2 x^2 + c_1 x + c_0=0,
\end{equation}
where
\begin{align*}
c_2(a,\theta) &= 2a(1-2\theta), \\
c_1(a,\theta) &= 3a^2-8a^2 \theta - 1, \\
c_0(a,\theta) &= 4a\theta(a^2-1).
\end{align*}
From above, we know $K$ satisfies either ($A_2$) or ($B_2$) so we should obtain a unique positive wave speed. Hence, it follows that \eqref{eq: wave_quadratic} has two real solutions, $x_+>0$ and $x_-<0$. Then $\mu(a,\theta)=\f{x_+(a,\theta)}$.
\subsection*{Threshold Analysis}\label{subsec: ex_thres}
 First of all, since $K$ is symmetric, 
 $$
 \integ{x}{-M_{2j}}[0]<K(x;a)>= \integ{x}{0}[N_{2j}]<K(x;a)>
 $$
so if $K$ satisfies $\mathcal{L}_{\infty}$, then for all $j\geq 1$,
$$
\integ{x}{0}[N_{2j}]<K(x;a)>=\integ{x}{-M_{2j}}[0]<K(x;a)> >\f{2}-\theta>\theta-\f{2}
$$
and $\mathcal{R}_{\infty}$ is also satisfied. Therefore, we only need to confirm $\mathcal{L}_{\infty}$. On that note, we observe that for all $a$, the zeros of $K$ on the left half plane are defined by $-M_{j}(a):=-(\arctan{\left(\f{a}\right)}+(j-1)\pi)$. Clearly
$$
\integ{x}{-M_{2(j+1)}}[-M_{2j}]<K(x;a)> >C(a) \e{-a M_{2j+1}}\integ{x}{-M_{2(j+1)}}[-M_{2j}]<-a\sin(x)+\cos(x)>=0.
$$
Therefore, the minimum value of $\integ{x}{-M_{2j}}[0]<K(x)>$ occurs when $j=1$ and recalling $\mathcal{L}_{\infty}$ requires 
\begin{align*}
\integ{x}{-M_{2j}}[0]<K(x;a)> >\f{2}-\theta \qquad \text{for } j\geq 1,
\end{align*}
it follows that $\mathcal{L}_\infty$ is satisfied if and only if $g(a)<\theta$, where 
\begin{align}
g(a):=\f{2}-\integ{x}{-M_2(a)}[0]<K(x;a)>.
\end{align}
\subsubsection*{Existence, Uniqueness, and Stability}
By \cref{thm: E_and_U}, we have demonstrated the existence and uniqueness of fronts in the region
\begin{equation}\label{eq: region_R}
\mathcal{R}=\lbrace(a,\theta)\in \R^+\times \left(0,\f{2}\right) \mid g(a)<\theta\rbrace.
\end{equation}
Finally, it is clear from the wave speed calculation that solutions are stable by \cref{cor: sin_cos_stab}. See Figure \ref{fig: Kelvin_fullclassification}.

\begin{remark}
There may be pairs $(a,\theta)\not\in \mathcal{R}$ where $\mathcal{L}_{\infty}$ is not satisfied, but fronts still exist. Such fronts will be unique and stable by the analysis above.
\end{remark}

\begin{figure}[H]
\centering
\includegraphics[width=130mm]{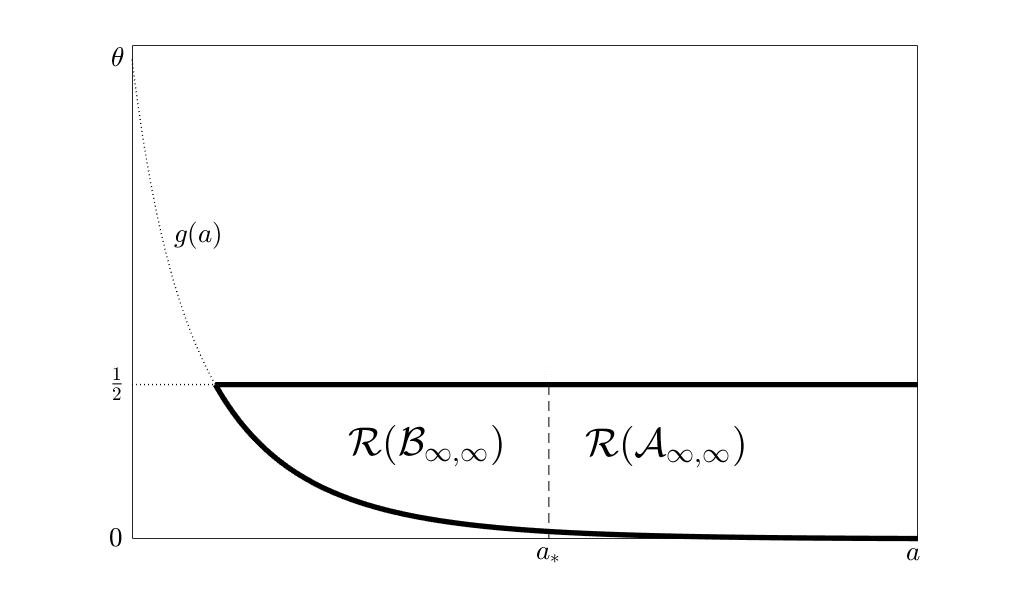}
\caption{Inside the dark lines: plot of region $\mathcal{R}$ on the $a\theta$ plane. All pairs $(a,\theta)\in \mathcal{R}$ lead to unique, stable traveling fronts. When $f(a_*)=0$, the kernel changes from class $\mathcal{B}_{\infty,\infty}$ to class $\mathcal{A}_{\infty,\infty}$.}
\label{fig: Kelvin_fullclassification}
\end{figure}

%SECTION: NUMERICALLY COMPUTED FAST PULSES

%COMMENTED OUT PULSE

\section{Existence of Traveling Pulses}
\label{sec: pulses}
With some exceptions, traveling wave fronts alone are typically of somewhat incomplete value since they unrealistically ignore metabolic processes that generate negative feedback. Their main value comes from their close connection with traveling pulses; studying the existence and stability of the front is the preparation for studying the existence and stability of pulses. Since the methodology is more unclear for multiple pulses, we will only focus on single pulses.

Mathematically, pulses can be categorized as slow or fast based on the wave speed.  Slow pulses are unstable perturbations of standing pulses, while fast stable pulses have a singular structure partially comprised of stable fronts and backs. For simplicity, we will ignore slow pulses since they are biologically less important.

Incorporating linear feedback into \eqref{eq:paper_main}, we obtain the following system of integral equations \cite{PintoandErmentrout-SpatiallyStructuredActivityinSynapticallyCoupledI.TravelingWaves,Pinto2005}:
\begin{align}
u_t + u + w & = \alpha \integ{y}{\R}[]<K(x-y)H(u(y,t-\f{c_0}|x-y|)-\theta)>,  \label{eq: pulse1}\\
w_t & = \epsilon(u-\gamma w), \label{eq: pulse2}
\end{align}
where $w$ is a slow leaking current and $0<\epsilon\ll1$ is a perturbation parameter. Biologically, $w$ represents phenomenon like spike frequency adaptation.

By a fast pulse solution to \eqref{eq: pulse1}-\eqref{eq: pulse2}, we mean a solution $(u,w) = (U_{pulse}(z),W_{pulse}(z))$ such that there exists a constant $\Z(\epsilon)>0$ such that $U_{pulse}(0)=U_{pulse}(\Z(\epsilon))=\theta, U_{pulse}(z) > \theta$ on $\OO{0}{\Z(\epsilon)}$, $U_{pulse}(z) < \theta$ on $\OO{-\infty}{0}\cup \OO{\Z(\epsilon)}{\infty}$, and $\ds \lim_{z\to \pm \infty} (U_{pulse}(z),W_{pulse}(z)) = (0,0)$. Here, $z = x + \mu(\epsilon) t$ is the traveling coordinate with unique fast wave speed $\mu(\epsilon) = \mu_{front} + \kappa(\epsilon)$, where $\kappa(\epsilon)\to 0$ as $\epsilon \to 0$. 

Since the system is singularly perturbed, we cannot simply set $\epsilon=0$ and obtain a solution that reasonably approximates a pulse. However, we can construct singular homoclinical orbits in phase space and argue that for $0<\epsilon \ll 1$, real pulses are close to singular pulses. Such a singular homoclinical orbit is comprised of a front, back, and two space curves. The front and back (which have the same wave speed) are understood to capture the fast dynamics, while the two space curves are where slow time dynamics occur. Naturally, as pointed out in \cite{PintoandErmentrout-SpatiallyStructuredActivityinSynapticallyCoupledI.TravelingWaves}, the subject of geometric singular perturbation theory may be a method for proving pulses exist. However, working out the details is quite nontrivial since \eqref{eq: pulse1}-\eqref{eq: pulse2} is not necessarily autonomous; applying geometric singular perturbation theory rigorously to \eqref{eq: pulse1}-\eqref{eq: pulse2} is still an open problem for general kernel functions. 

We remark that many kernels presented in this paper (as evident by the examples given) allow us to convert \eqref{eq: pulse1}-\eqref{eq: pulse2} into a system of real, nonlinear, autonomous PDEs. The dynamics then reside in a more familiar setting to possibly apply techniques like the celebrated Exchange Lemma \cite{JonesKopell} in the setting of center-stable and center-unstable manifold theory. The main setback in our problem is that the Heaviside firing rate function creates phase space dynamics with discontinuities; it is important that this issue is properly dealt with in order to apply such a technical result. Some promising partial results have been obtained in \cite{Faye2013,Faye2015} for related models with smooth firing rate functions.

Pulses have also been proven to exist using more direct computational tools like implicit function theorem, as was done in \cite{Pinto2005} when $K(x) = \f{\rho}{2}\e{-\rho |x|}$. Although this is a powerful result, the exact structure of $K$ played an important role. Again, it is not entirely clear how to rigorously prove the existence of pulses to \eqref{eq: pulse1}-\eqref{eq: pulse2} for general kernel functions.

\subsection*{Numerically Computed Fast Pulses}
In this subsection, we first derive formulas for fast pulses formally and then calculate pulses with example kernels $K_1, K_2$, and $K_3$ respectively. Using phase space plots, we show that our solutions are real pulses. To simplify the discussion, we assume $c_0 = \infty$.

We wish to find traveling pulse solutions to \eqref{eq: pulse1}-\eqref{eq: pulse2} when $0<2\theta < \alpha < \f{(1+\gamma)\theta}{\gamma}$,  $0<\epsilon \ll 1$ and $K$ is a kernel such that a unique front is produced when $\epsilon = 0$. Assuming a pulse exists with $\mu(\epsilon)$ and $\mathcal{Z}(\epsilon)$ to be determined, \eqref{eq: pulse1}-\eqref{eq: pulse2} reduces to
\begin{align}
\mu(\epsilon) U' + U + W & = \alpha \integ{x}{z-\mathcal{Z(\epsilon)}}[z]<K(x)>, \label{eq: fastpulseformal1} \\
\mu(\epsilon) W' & = \epsilon(U-\gamma W), \label{eq: fastpulseformal2}
\end{align}
which is easily solved using elementary techniques. The solution takes on the form
\begin{align}
U_{pulse}(z) & = \f{\alpha \gamma}{1+\gamma}\integ{x}{z-\mathcal{Z}(\epsilon)}[z]<K(x)> -\alpha \integ{x}{-\infty}[z]<C(x-z,\mu(\epsilon),\epsilon)\left[K(x)-K(x-\mathcal{Z}(\epsilon))\right]>, \label{eq: Upulseformal} \\
U_{pulse}'(z) & = \alpha \integ{x}{-\infty}[z]<C_x(x-z,\mu(\epsilon),\epsilon)\left[K(x)-K(x-\mathcal{Z}(\epsilon))\right]> \label{eq: Uderivpulseformal} \\ 
W_{pulse}(z) & = \f{\alpha}{1+\gamma}\integ{x}{z-\mathcal{Z}(\epsilon)}[z]<K(x)> -\epsilon \alpha \integ{x}{-\infty}[z]<D(x-z,\mu(\epsilon),\epsilon)\left[K(x)-K(x-\mathcal{Z}(\epsilon))\right]>, \label{eq: Wpulseformal}
\end{align}
where 
\begin{align*}
C(x,\mu(\epsilon),\epsilon) & = \f{\omega_1(\epsilon)-\omega_2(\epsilon)}\left[\f{1-\omega_2(\epsilon)}{\omega_1(\epsilon)}\e{\f{\omega_1(\epsilon)}{\mu(\epsilon)}x}-\f{1-\omega_1(\epsilon)}{\omega_2(\epsilon)}\e{\f{\omega_2(\epsilon)}{\mu(\epsilon)}x} \right], \\
D(x,\mu(\epsilon),\epsilon) & = \f{\omega_1(\epsilon)-\omega_2(\epsilon)}\left[-\f{\omega_1(\epsilon)}\e{\f{\omega_1(\epsilon)}{\mu(\epsilon)}x}+\f{\omega_2(\epsilon)}\e{\f{\omega_2(\epsilon)}{\mu(\epsilon)}x} \right], \\
\omega_1(\epsilon) & = \f{1+\gamma \epsilon + \sqrt{(1-\gamma \epsilon)^2-4\epsilon}}{2}, \\
\omega_2(\epsilon) & = \f{1+\gamma \epsilon - \sqrt{(1-\gamma \epsilon)^2-4\epsilon}}{2}.
\end{align*}
The parameters $\omega_1(\epsilon)$ and $\omega_2(\epsilon)$ are the eigenvalues associated with the coefficient matrix ${A(\epsilon) =
\begin{pmatrix}
1 & 1 \\
-\epsilon & \gamma \epsilon
\end{pmatrix}}$. We assume that $\epsilon$ and $\gamma$ are sufficiently small so that both eigenvalues are positive.

Similar to the front, we must use the compatibility equations $U_{pulse}(0) = U_{pulse}(\mathcal{Z}(\epsilon)) = \theta$ to solve for $\mu(\epsilon)$ and $\mathcal{Z}(\epsilon)$, provided they exist. Then we must verify that the formal solutions lead to real traveling pulse solutions that satisfy the threshold requirements. As mentioned above, there are typically two solution pairs, $(\mu_{fast}(\epsilon),\mathcal{Z}_{fast}(\epsilon))$ and $(\mu_{slow}(\epsilon),\mathcal{Z}_{slow}(\epsilon))$, leading to stable and unstable pulses respectively \cite{PintoandErmentrout-SpatiallyStructuredActivityinSynapticallyCoupledI.TravelingWaves,Pinto2005}. Since it is not entirely clear how to rigorously prove the existence of pulses for all kernels studied in this paper, we simply calculate $(\mu_{fast}(\epsilon),\mathcal{Z}_{fast}(\epsilon))$ numerically; then we compare the solutions to the singular solutions and argue that real fast pulses exist.

In Figures \ref{fig: pulse_ex1}, \ref{fig: pulse_ex2}, and \ref{fig: pulse_ex3}, we see evidence that fast pulses exist for our kernel classes.

\begin{figure}[H]
\centering
\includegraphics[width=150mm]{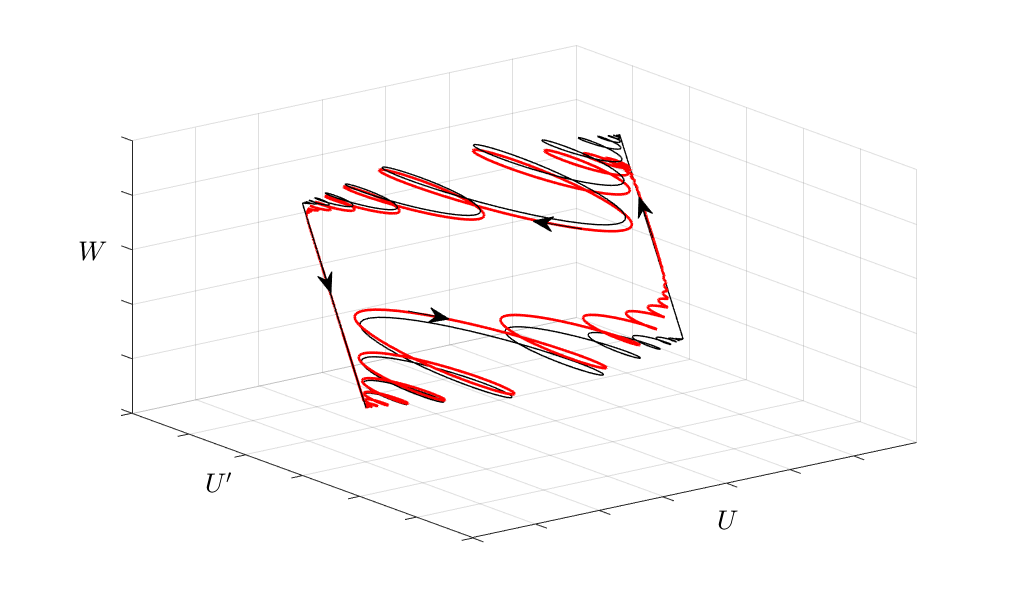}
\caption{Red: Phase space portrait of a fast traveling pulse with kernel $K_1$, $\alpha = 1$, $\theta = 0.4$, $\epsilon=\gamma = 0.001$. Black: Corresponding singular homoclinical orbit when $\epsilon = 0$.}
\label{fig: pulse_ex1}
\end{figure} \rule{0ex}{0ex}

\begin{figure}[H]
\centering
\includegraphics[width=150mm]{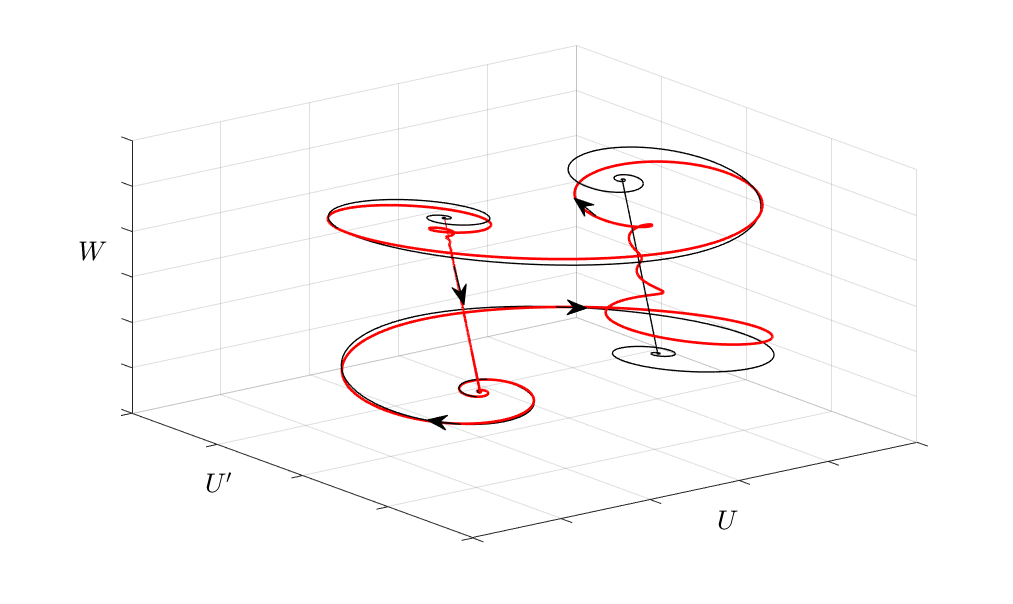}
\caption{Red: Phase space portrait of a fast traveling pulse with kernel $K_2$, $\alpha = 1$, $\theta = 0.4$, $\epsilon=\gamma = 0.001$. Black: Corresponding singular homoclinical orbit when $\epsilon = 0$.}
\label{fig: pulse_ex2}
\end{figure} \rule{0ex}{0ex}

\begin{figure}[H]
\centering
\includegraphics[width=150mm]{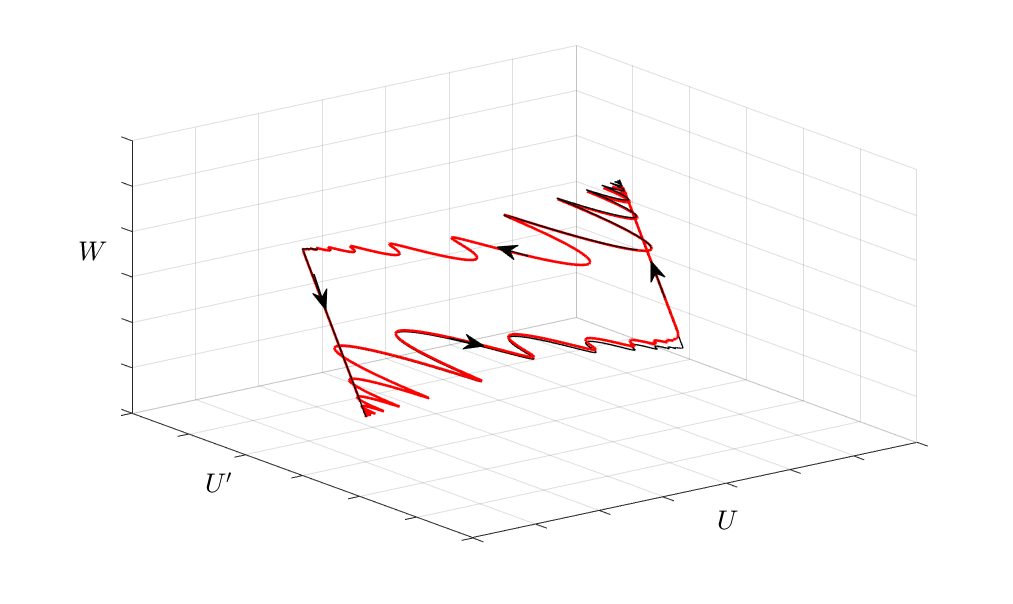}
\caption{Red: Phase space portrait of a fast traveling pulse with kernel $K_3$, $\alpha = 1$, $\theta = 0.4$, $\epsilon=\gamma = 0.001$. Black: Corresponding singular homoclinical orbit when $\epsilon = 0$.}
\label{fig: pulse_ex3}
\end{figure} \rule{0ex}{0ex}

Note that all plots were constructed in Matlab. For the fast pulses, the parameters $\mu(\epsilon)$ and $\mathcal{Z}(\epsilon)$ were computed using the Optimization Toolbox.

\section*{Discussion}

In this paper, we proved the existence and uniqueness of traveling front solutions to \eqref{eq:paper_main} for a wide range of oscillatory kernel classes. To establish uniqueness, we used a new technique to show speed index functions have at most one critical point. In order to make sure that the fronts crossed the threshold exactly once, we established left and right half plane conditions, following the techniques in \cite{LijunZhangetal, LvWang}. In \cref{sec: spectral_stability}, we studied the spectral stability of our fronts using the Evans function approach. Combining our results, we were able to fully classify the existence, uniqueness, and stability of a well-studied one parameter kernel type that crosses the $x$-axis countably many times.

Our results have inspired some interesting open problems: How many of our kernel classes lead to stable versus unstable fronts? Can a meaningful bifurcation criteria be formulated? What is the impact of threshold noise? Can we obtain similar results for kernels with heterogeneities? If so, we would make further mathematical advancements with a meaningful biological connection.

Lastly, in \cref{sec: pulses}, we reviewed the connection between fronts and fast pulses that arise in singularly perturbed integral differential equations. We derived pulses formally and numerically calculated fast pulses for all three example kernels used throughout the paper. Phase space portraits show that indeed, fast pulses are small perturbations of singular homoclinical orbits when $0<\epsilon\ll 1$. We desire a rigorous technique for proving the existence and uniqueness of fast pulses when $K$ crosses the $x$-axis at most countably many times. How often do fast pulses exist for kernels in this paper? Can we obtain fronts for smoothed Heaviside firing rate functions? If so, we may be able to apply typical methods from perturbation theory.

\section*{Acknowledgements}

The author would like to thank his advisor, Linghai Zhang, for offering feedback for the manuscript and the Lehigh University College of Arts and Sciences for a generous summer research fellowship.
%bibliography
\bibliographystyle{siam}
\bibliography{bibliography}
\end{document}